\documentclass[a4paper]{article}

\usepackage[all]{xy}\usepackage[latin1]{inputenc}        
\usepackage[dvips]{graphics,graphicx}
\usepackage{amsfonts,amssymb,amsmath,color,mathrsfs, amstext}
\usepackage{amsbsy, amsopn, amscd, amsxtra, amsthm,authblk, enumerate}
\usepackage{upref}
\usepackage{geometry}
\usepackage[displaymath]{lineno}

\usepackage{float}
\usepackage{yhmath}
\usepackage[colorlinks,
            linkcolor=red,
            anchorcolor=red,
            citecolor=red
            ]{hyperref}

\numberwithin{equation}{section}

\def\e{\varepsilon}
\def\epsilon{\varepsilon}

\DeclareMathOperator{\supp}{supp}

\DeclareMathOperator{\loc}{loc}

\newcommand{\tcb}[1]{\textcolor{blue}{#1}}

\newcommand{\inner}[2]{\langle {#1}, {#2}\rangle}
\newcommand{\bbE}{\mathbb{E}}
\newcommand{\bbR}{\mathbb{R}}
\newcommand{\bbP}{\mathbb{P}}
\newcommand{\bfP}{\mathbf{P}}
\newcommand{\bfX}{\mathbf{X}}

\newcommand{\cH}{\mathcal{H}}

\newcommand{\calk}{\mathcal{K}}
\newcommand{\oino}{\omega\in\Omega}

\newtheorem{theorem}{Theorem}[section]
\newtheorem*{theorem*}{Theorem}
\newtheorem{lemma}{Lemma}[section]
\newtheorem{proposition}{Proposition}[section]
\newtheorem*{proposition*}{Proposition}

\newtheorem*{corollary*}{Corollary}
\newtheorem{definition}{Definition}[section]
\newtheorem*{definitions*}{Definitions}
\newtheorem*{conjecture*}{\bf Conjecture}

\newtheorem*{example*}{\bf Example}
\theoremstyle{remark}
\newtheorem{remark}{\bf Remark}[section]

\numberwithin{equation}{section}

\begin{document}

\title{On the mean field limit for Brownian particles with Coulomb interaction in 3D}

\author[1]{Lei Li\thanks{leili2010@sjtu.edu.cn}}
\author[2]{Jian-Guo Liu\thanks{jliu@phy.duke.edu}}
\author[3]{Pu Yu\thanks{qinghefeng@pku.edu.cn}}
\affil[1]{School of Mathematical Sciences, Institute of Natural Sciences, MOE-LSC, Shanghai Jiao Tong University, Shanghai, 200240, P. R. China.}
\affil[2]{Department of Mathematics and Department of Physics, Duke University, Durham, NC 27708, USA.}
\affil[3]{School of Mathematical Sciences, Peking University, Beijing, 100871, P. R. China.}

\date{}
\maketitle

\begin{abstract}
In this paper, we consider the mean field limit of Brownian particles with Coulomb repulsion in 3D space using compactness. Using a symmetrization technique, we are able to control the singularity and prove that the limit measure almost surely is a weak solution to the limiting nonlinear Fokker-Planck equation. Moreover, by proving that the energy almost surely is bounded by the initial energy, we improve the regularity of the weak solutions. By a natural assumption, we also establish the weak-strong uniqueness principle, which is closely related to the propagation of chaos.
\end{abstract}

\section{Introduction}

There are many phenomena in natural and social sciences that are related to interacting particles \cite{vicsek1995novel,cucker2007emergent,horstmann03,liuxin2000}.  An effective method for studying these large and complex systems where small individuals interact with each other is  the mean field approximation \cite{stanley1971, georges1996, lasry2007, jabin2017}. In this approximation, the effect of surrounding particles is approximated by a consistent averaged force field so that we have a one body problem. The mean field approximation naturally applies to 
the kinetic theory where the macroscopic properties of gases are studied \cite{grad1949,kac1956,jabin2014,serfaty2018mean}. A desired property in the mean field limit is the so-called ``propagation of chaos" \cite{kac1956,mckean1967,meleard1987,sznitman1991}.
Roughly speaking, starting with a chaotic initial configuration where the particles are from independent copies of the initial state,  the statistical correlation between finite groups of particles vanishes at a later (fixed) time as the number of particles goes to infinity. In other words,  the particles reduce to independent copies of nonlinear Markov processes. 

In this paper, we are interested in the mean field limit of Brownian particles with Coulomb interaction in three dimensional space. More precisely, we consider the $N$ particle system
\begin{gather}\label{model}
dX_t^{i,N}=-\frac{1}{N}\nabla_{x_i}\cH(X_t^{1,N},\ldots, X_t^{N,N})\,dt+\sqrt{2}\,dB_i^N,~i=1,\ldots, N
\end{gather}
where
\begin{gather}
\cH(x_1,\ldots, x_N)=\sum_{i,j: i<j}g(x_i-x_j)=\frac{1}{2}\sum_{i,j: i\neq j}g(x_i-x_j),
\end{gather}
with $g$ being the interaction potential. This system can be regarded as the overdamped limit of Langevin equations \cite{coffey2012langevin} so that $dX$'s are from the friction terms while $\frac{1}{N}\nabla_{x_i}\cH$ can be thought as the interacting forces. Hence, in the remaining part of the paper, we will call this the ``interacting forces'', though it may have other interpretations in some applications. The scaling $1/N$ appears because we desire to have a total mass or charge to be $O(1)$ so that there is a mean field limit as $N\to \infty$. The Brownian motion is not scaled since the strength is determined by the temperature instead of the mass of the particle. Of course, if one desires to consider other scaling regimes, there may be factors depending on $N$ for both terms.
The initial values $\{X_0^{i,N}\}$ are i.i.d from some given density $\rho_0$. Also, $\{B_i^N\}_{i=1}^N$ are independent $d$-dimensional standard Brownian motions.
Hence, the $N$-particle system is totally determined by $(X_0^{1,N},\ldots, X_0^{N,N}, B_1^N,\ldots, B_N^N)$. It is standard by Kolmogorov extension theorem \cite{tao2011,durrett2010} that there exists a probability space $(\Omega, \mathcal{F}, \bbP)$ so that all the random variables $\{(X_0^{1,N},\ldots, X_0^{N,N}, B_1^N,\ldots, B_N^N)\}_{N=1}^{\infty}$ are on this probability space and they are all independent.
Clearly, if we identify $B_i^N$ for different $N$'s the law of $X_t^{i,N}$ is unchanged. Hence, we will drop the index $N$ for the Brownian motions from now on. Moreover, we use $\bbE$ to mean the expectation under $\bbP$.  If the interaction potential is given by
\begin{gather}
g(x)=
\begin{cases}
-\frac{1}{2\pi}\log |x|, & d=2 ,\\
C_d|x|^{-(d-2)}, & d\ge 3,
\end{cases}
\end{gather}
where 
\begin{gather}
C_d=\frac{1}{d(d-2)\pi}\left(\frac{\Gamma(d)}{\Gamma(\frac{d}{2})}\right)^{\frac{2}{d}}
\end{gather}
is chosen such that $-\Delta g=\delta_0$ holds in the distributional sense, then the interaction is called Coulomb repulsive interaction. Moreover, we define the Coulomb repulsive force as 
\begin{gather}
F(x)=-\nabla g(x)= (d-2) C_d \frac{x}{|x|^{d}}.
\end{gather}
We will particularly focus on $d=3$ case but some discussions are performed for general $d$.

Our goal is to show that as $N\to \infty$, the empirical measure
\begin{gather}\label{eq:empirical}
\mu^N=\frac{1}{N}\sum_{i=1}^N\delta_{X^{i,N}}
\end{gather}
(as a random variable taking values in $\bfP(C([0, T]; \bbR^d))$) converges in law to a random measure $\mu$, whose density is almost surely a  solution of the nonlinear Fokker-Planck equation
\begin{equation}\label{eq:pnp}
    \partial_t\rho=\Delta\rho+\nabla\cdot(\rho\nabla(g*\rho)), ~~\rho_t|_{t=0}=\rho_0
\end{equation}
provided that the initial empirical measure converges weakly to the initial data $\rho_0$.  The meaning of ``solution" here is the weak solution, which will be clarified later (Definition \ref{weaksol:def}).
In fact, we have the following result.
\begin{theorem*}[Informal version of Theorem \ref{asweak2}]
When $d=3$, under suitable conditions of the initial data $\rho_0$, any limit point $\mu$ of the empirical measure $\mu^N$ under the topology of convergence in law in $\bfP(C([0,T];\bbR^d))$ has a density $\rho$ a.s., and $\rho$ is a weak solution to \eqref{eq:pnp} a.s. in the sense of Definition \ref{weaksol:def}.
\end{theorem*}
If the solution $\rho$ is proved to be unique, so that the limit measure $\mu$ is deterministic, then we have the propagation of chaos (see \cite[Proposition 2.2]{sznitman1991}).  However, the regularity of the weak solution in Definition \ref{weaksol:def} is limited and it is very challenging to show the uniqueness of the weak solutions under these conditions even though the initial data $\rho_0$ is very good. On the other hand, it is standard to show that if $\rho_0\ge 0$, $\rho_0\in L^1(\mathbb{R}^d)\cap H^{m}(\mathbb{R}^d)$ with some $m>d/2$, the equation has a unique global strong solution $\rho\in C([0,T];L^1(\bbR^d))\cap C([0,T];H^m(\bbR^d))$ with $\rho\ge 0$. Moreover, $\|\rho\|_1=\|\rho_0\|_1$ and $\rho\in C^{\infty}((0,\infty), H^s)$ for all $s\ge 0$. See Appendix \ref{app:strongsoln} (Proposition \ref{pro:strongunique}--\ref{pro:nonblowup}) for reference. Hence, one desires to improve the regularity of the weak solutions so that one can eventually show that the "weak solution" by the limit measure is the same as the strong solution, which is one common way to establish the propagation of chaos.  This is known to be the "weak-strong uniqueness principle" \cite{brenier2011,feireisl2012,duerinckx2016,fischer2017}.

As a second main result of this work, we show that the energy almost surely is bounded by the initial energy so that we can improve the regularity of the limiting weak solutions (see Proposition \ref{pro:uniformboundenergy} and \ref{pro:improvedweak}). Together with this, the extra assumption that the weak solution $\rho\in L_{\loc}^2([0, T]; L^2(\bbR^3))$ can imply the weak-strong uniqueness principle.  

\begin{theorem*}[Informal version of Proposition \ref{pro:uniformboundenergy} and Proposition \ref{prop:uniquenessresult}]
With suitable assumptions on the initial data $\rho_0$, for general dimension $d\ge 3$, for any limit point $\mu$ of $\mu^N$, the energy is bounded by the initial energy almost surely
\[
\iint_{\bbR^d\times\bbR^d} g(x-y)\mu_t(dx)\mu_t(dy)
\le \iint_{\bbR^d\times\bbR^d} g(x-y)\rho_0(x)\rho_0(y)\,dxdy.
\]
Consequently, the density $\rho\in L^{\infty}(0, T; H^{-1})$ and $\nabla (g* \rho)\in L^{\infty}(0, T; L^2)$ almost surely.
Lastly, for $d=3$, if one further assumes such a weak solution $\rho\in L_{\loc}^2(0, T; L^2(\bbR^3))$, then it must be the unique strong solution.
\end{theorem*}
We will explain in Section \ref{sec:3dchaos} that the $L_{\loc}^2(0, T; L^2(\bbR^3))$ assumption makes perfect physical sense due to the energy dissipation. However, rigorous justification is not easy.
Combining these two results, we obtain a condition for the propagation of chaos.
\begin{theorem*}[Informal version of Theorem \ref{thm:propagationchaos}]
 Consider $d=3$. With suitable assumptions on the initial data $\rho_0$, if the density $\rho$ for any limit point $\mu$ of the empirical measures $\mu^N$ satisfies $\mathbb{E}\int_0^T\int_{\bbR^3}\rho^2\,dxdt<\infty$, then there is propagation of chaos.
\end{theorem*}

The tool we use to establish the above results is the compactness method based on entropy and Fisher information estimates. In \cite{liuyang2016}, the propagation of chaos result for $d=2$ case is proved using compactness through a self-consistent martingale problem. The proof needs to control singularity with $(d-1)$th order using Fisher information from \cite{fournier2014} (see Lemma \ref{fisherIII} for a slightly generalized form). However, the proof there cannot be applied directly to $d\ge 3$ cases. (As will be remarked in section \ref{sec:limitmeasure}, the compactness method based on Fisher information seems not to work for $d\ge 4$ cases, and new tools should be developed to tackle this problem.) What we do is to use certain symmetrization to reduce the singularity in the third term of \eqref{eq:pnp} from $d-1$ to $d-2$. Using this trick and the estimates of Fisher information, we show that the limit measure almost surely is a weak solution to the limiting nonlinear Fokker-Planck equation \eqref{eq:pnp} for  $d=3$ (Theorem \ref{asweak2}). This is the first main result in this work. As already mentioned, the weak-strong uniqueness principle is only established by assuming that the density of the limit measure almost surely is in $L_{\loc}^2(0, T; L^2(\bbR^3))$ (see Proposition \ref{prop:uniquenessresult}). Though physically significant, the justification seems hard.

Let us mention some related references, which by no means are exhaustive. 
In \cite{hauray2009,duerinckx2016,serfaty2018mean}, the mean field limit problems for particle systems without Brownian motions with various interaction kernels have been established. In particular, in \cite{serfaty2018mean}, Serfaty and Duerinckx established the results for particles with Coulomb interaction even for $d\ge 3$.  When Brownian motions are present, we have stochastic systems \cite{osada1987,cattiaux2008,fournier2014,godinho2015,liuyang2016,jabinquantitative,berman2016}. 
In \cite{cattiaux2008}, propagation of chaos was proved uniformly in time when the interaction kernel is regular enough and a confining potential is present. In \cite{liuyang2016}, the propagation of chaos for 2D Coulomb interaction was proved using nonlinear martingale problems. 
In \cite{jabinquantitative}, the propagation of chaos for $W^{-1,\infty}$ kernels has been established, and this include the kernels considered in \cite{fournier2014,liuyang2016}. By estimating the relative entropy, they found the convergence rate of propagation of chaos for some models. However, the 3D Coulomb kernel is not included in their model so their method does not apply.

The rest of the paper is organized as follows. In section \ref{sec:setup}, we review and prove some basic results for Fisher information of probability measures and $N$ particle systems. In particular, the empirical measures of the $N$ particle systems are tight so that any subsequence has a further converging subsequence to some limiting measure. Also, there are uniform estimates of the Fisher information. 
In section \ref{sec:limitmeasure}, using a symmetrization technique together with the Fisher information estimate, we show that the limit measure almost surely is a weak solution to the nonlinear Fokker-Planck equation \eqref{eq:pnp}. In section \ref{sec:3dchaos}, we establish the weak-strong uniqueness principle based on the assumption $\rho\in L_{\loc}^2(0, T; L^2(\bbR^3))$, and remark on the propagation of chaos.
In appendices \ref{app:strongsoln} and \ref{app:missingpf}, we provide the notes for strong solutions and missing proofs for reference.

\section{Setup and existing results}\label{sec:setup}
In this section, we first recall some basic properties of Fisher information and extend the estimates in \cite{fournier2014} to high-dimensional cases. Then we give an alternative proof for the well-posedness of the system \eqref{model}. Finally we present the results of tightness of the empirical measures in \cite{liuyang2016}.

\subsection{Entropy and Fisher information of probability measures}
We begin with the definition of Fisher information. For any probability measure $f\in\mathbf{P}((\bbR^d)^k)$, we recall that the entropy and Fisher information are defined respectively by
\begin{gather}\label{eq:entropy}
H(f):=
\begin{cases}
\int_{\bbR^{kd}}\rho\log\rho\,dx,&\text{if }f=\rho\,dx,\\
+\infty &\text{otherwise};
\end{cases}
I(f):=
\begin{cases}
\int_{\bbR^{kd}}\frac{|\nabla\rho|^2}{\rho} dx,&\text{if }f=\rho\,dx,\\
+\infty &\text{otherwise}.
\end{cases}
\end{gather}
We also introduce the normalized entropy and Fisher information for $f\in\bfP((\bbR^d)^k)$:
\begin{gather}\label{eq:normalizedentropy}
H_k(f):=\frac{1}{k}H(f), ~~
I_k(f):=\frac{1}{k}I(f).
\end{gather}
The normalized version is introduced so that $H_k(f^{\otimes k})=H_1(f)$ and $I_k(f^{\otimes k})=I_1(f)$ for $f\in\mathbf{P}(\bbR^d)$, which is convenient for the mean field limit discussion. We remark that the notations we use here are different from those in \cite{hauray2014}, where they use $H$ to mean the normalized version while $H_j$ is the unnormalized version. In the following discussion, we sometimes use $I_k(\rho)$ and $H_k(\rho)$ to represent $I_k(f)$ and $H_k(f)$ (or $I(\rho), H(\rho)$ to represent $I(f)$ and $H(f)$).

We denote the set of all symmetric probability measures on $(\bbR^d)^k$ by $\bfP_{sym}((\bbR^d)^k)$. By "symmetric",  we mean the measure stays unchanged under the pushforward corresponding to any permutation of the $k$ copies of $\bbR^d$. If the distribution of some $k$ particle system is symmetric for all time, then the system is said to be exchangeable. Recall that for the joint probability distribution $F\in \bfP(\mathcal{X}^k)$ of $k$ random variables taking values in $\mathcal{X}$, the marginal distribution of $X^{i_1},\ldots, X^{i_r}$ with $\{i_1,\ldots, i_r\}\subset \{1,\ldots, N\}$ is defined by 
\begin{gather}
F_{i_1,\cdots, i_r}:=\int_{\mathcal{X}^{k-r}}F dx^{i_{r+1}}\ldots dx^{i_{k}},
\end{gather}
where $\{i_{r+1},\ldots, i_k\}=\{1,\ldots, N\}\setminus \{i_1, \ldots, i_r\}$. If $F$ is symmetric, the marginal distributions are the same for different choices of $i_1,\ldots, i_r$ and this will be called the $r$-marginal distribution later, denoted by $F^{(r)}$.

Below, we list out some standard properties of the entropy and Fisher information.
\begin{lemma}[\cite{hauray2014}, Lemma 3.3]\label{entropy1}
We have the following super-additivity of entropy:
\begin{enumerate}
\item Suppose all the one marginal distributions of $F\in\bfP((\bbR^d)^k)$ are the same, denoted by $f\in\bfP(\bbR^d)$. If $H_1(f)<\infty$, then $H_k(F)\ge H_1(f)$. The equality holds if and only if $F=f^{\otimes k}$.

\item Consider $F\in\bfP((\bbR^d)^k)$. Then the un-normalized entropies satisfy 
\begin{gather}
H(F)\ge H(F_{i_1,\cdots, i_r})+H(F_{i_{r+1},\cdots, i_k}).
\end{gather}
 The equality holds if and only if $F=F_{i_1,\cdots, i_r} \otimes F_{i_{r+1},\cdots, i_k}$ where $r \in\{ 1, ..., k-1\}$.
\end{enumerate}
\end{lemma}

\begin{proof}
We briefly give the proof below. By Jensen's inequality, we have $\int_E g\log \frac{h}{g}\,dx\le 0$ for any probability densities $g,h$ on a Polish space $E$. Hence,
\begin{equation}\label{hauary3.3}
  \int_E g\log g\,dx\ge \int_E g\log h\,dx.
\end{equation}
The equality holds if and only if $g=h$, a.e.. Now we take $E=(\bbR^d)^k$. The first part of the claim follows by taking $g=F$ and $h=f^{\otimes k}$ in \eqref{hauary3.3}. And the second part follows by taking $g=F$ and $h=F_{i_1,\cdots, i_r} \otimes F_{i_{r+1},\cdots, i_k}$.
\end{proof}

\begin{lemma}[{\cite[Theorem 3]{carlen1991}}]\label{fisherI}
Suppose $F\in\bfP((\bbR^d)^k)$. Then the non-normalized Fisher information satisfies
\begin{equation}
    I(F)\ge I(F_{i_1,\cdots, i_r})+I(F_{i_{r+1},\cdots, i_k}).
\end{equation}
The equality holds if and only if $F=F_{i_1,\cdots, i_r} \otimes F_{i_{r+1},\cdots, i_k}$.
\end{lemma}

From Lemma \ref{entropy1} and Lemma \ref{fisherI}, for $f\in\bfP_{sym}((\bbR^d)^k)$ with $j$th marginal distribution $f^{(j)}$, where $k=qj+r$, $q,r\in\mathbb{Z}$,$q\ge0,0\le r\le j-1$, one then has
\begin{equation}\label{eq:entropyrelations}
    kI_k(f)\ge qjI_j(f^{(j)})+rI_r(f^{(r)}),\quad kH_k(f)\ge qjH_j(f^{(j)})+rH_r(f^{(r)}).
\end{equation}
Moreover, \cite[Lemma 3.7]{hauray2014} shows that that
\begin{equation}\label{fishersym}
  I_j(f^{(j)})\le I_k(f),
\end{equation}
i.e., for symmetric probability measures, the normalized Fisher information for marginal distributions $f^{(j)}$ can always be bounded by $I_k(f)$.

Since the entropy might be negative, we do not have $H_j(f^{(j)})\le H_k(f)$ and $H_j(f^{j})\le (k/(qj))H_k(f)$. To resolve this, we note the following lemma, which gives a lower bound for the entropy by moments of $f$.
\begin{lemma}[\cite{hauray2014}, Lemma 3.1]\label{entropy2}
For any $p,\lambda>0$, there exists a constant $C_{p,\lambda}\in\bbR$ such that for any $k\ge 1$, $F\in\bfP((\bbR^d)^k)$ with
$$M_p(F)=\int_{(\bbR^{d})^k}\frac{1}{k}\sum_{i=1}^k(|x_i|^2+1)^{\frac{p}{2}}F(dx)<\infty,$$
one has
\begin{equation}\label{lbentropy}
    H_k(F)\ge-C_{p,\lambda}-\lambda M_p(F)
\end{equation}
\end{lemma}
Combining equations \eqref{eq:entropyrelations} and \eqref{lbentropy}, one gets a control of $H_j(f^{j})$ in terms of $H_k(f)$ as follow.
\begin{equation}
    H_j(f^{(j)})\le\frac{k}{qj}H_k(f)+\frac{r}{qj}(\lambda M_p(f^{(r)})+C_{p,\lambda}).
\end{equation}

Next we extend the estimates in \cite{fournier2014} to high-dimensional cases.
\begin{lemma}\label{fisherII}
Let $d\ge 3$.
For any probability density $f$ in $\bbR^d$ with finite Fisher information $I(f)$, one has
\begin{equation}\label{fisher:wq}
    \forall q\in \left[1,\frac{d}{d-1}\right], \|\nabla f\|_{L^q(\bbR^d)}\le C_{q,d}I(f)^{\frac{d+1}{2}-\frac{d}{2q}};
\end{equation}
\begin{equation}\label{fisher:lp}
    \forall p\in \left[1,\frac{d}{d-2} \right], \| f\|_{L^p(\bbR^d)}\le C_{p,d}I(f)^{\frac{d}{2}(1-\frac{1}{p})}.
\end{equation}
If $d=2$, then \eqref{fisher:wq} holds for $q\in \left[1,2\right)$ while \eqref{fisher:lp} holds for $ p\in \left[1,+\infty \right)$.
\end{lemma}

\begin{proof}
We start from \eqref{fisher:wq}. By H\"{o}lder's inequality
\begin{equation}\label{fisherII.1}
\|\nabla f\|_q^q\le \left(\int_{\bbR^d}\frac{|\nabla f|^2}{f}dx \right)^{\frac{q}{2}}\|f\|_{\frac{q}{2-q}}^{\frac{q}{2}}.
\end{equation}
For $1\le q\le\frac{d}{d-1}$, we use the interpolation along with Sobolev's inequality
\begin{equation}\label{fisherII.2}
\|f\|_{\frac{q}{2-q}}\le\|f\|_1^{1-\theta}\|f\|_{\frac{dq}{d-q}}^{\theta}\le C_{q,d}\|\nabla f\|_q^{\theta},
\end{equation}
where $\theta$ is given by $\frac{2-q}{q}=\frac{d-q}{dq}\theta+(1-\theta)$. Note that $f$  is a probability density. Plugging \eqref{fisherII.1} into \eqref{fisherII.2}, we get \eqref{fisher:wq}.

Now for $1\le p\le\frac{d}{d-2}$, we can find some $1\le r\le\frac{d}{d-1}$ satisfying $p=\frac{r}{2-r}$. Then by \eqref{fisherII.2} and \eqref{fisher:wq} we can easily obtain \eqref{fisher:lp}.
\end{proof}

The following lemma is a slight generalization of those in \cite{hauray2014} to higher dimension, which is important to control some singular integrals using Fisher information.
\begin{lemma}\label{fisherIII}
Suppose $(X_1, X_2)$ is a random variable with density $F$ in $\bbR^d\times\bbR^d$. Assume that $F$ has finite Fisher information $I(F)$. 
\begin{enumerate}
\item For any $0<\gamma<2$ and $\frac{\gamma}{d}<\beta\le\frac{2}{d}$, there exists $C_{\gamma,\beta}$ such that
\begin{equation}\label{fisher:exp1}
    \bbE[|X_1-X_2|^{-\gamma}]=\int_{\bbR^d\times\bbR^d}\frac{F(x_1,x_2)}{|x_1-x_2|^{\gamma}}dx_1dx_2\le C_{\gamma,\beta}\left( I(F)^{\frac{\beta d}{2}}+1 \right).
\end{equation}
Moreover, for any $\epsilon>0$,  the following estimate holds:
\begin{equation}\label{fisher:exp2}
    \int_{|x-y|<\epsilon}\frac{F(x,y)}{|x-y|^{\gamma}}dxdy\le C_{\gamma,\beta}\epsilon^{d\beta-{\gamma}}I(F)^{\frac{d\beta }{2}}.
\end{equation}

\item For $d\ge 3$ and $\gamma=2$, it also holds that
\begin{gather}
 \bbE[|X_1-X_2|^{-2}]=\int_{\bbR^d\times\bbR^d}\frac{F(x_1,x_2)}{|x_1-x_2|^{2}}dx_1dx_2
 \le \frac{C}{(d-2)^2}I(F).
 \end{gather}

\end{enumerate}
\end{lemma}

\begin{proof}
Set $Y_1=\frac{X_1-X_2}{\sqrt{2}}$, $Y_2=\frac{X_1+X_2}{\sqrt{2}}$ and denote the joint distribution of $(Y_1,Y_2)$ by $\tilde{F}(y_1,y_2)$. Then $I(F)=I(\tilde{F})$. Denote the density of $Y_1$ by $\tilde{f}$. From the super-additivity property of Fisher information (Lemma \ref{fisherI}), we see that $I(\tilde{f})\le I(\tilde{F})=2I_2(F)$. 

1. By simple computation:
\begin{equation}\label{fisherIII.1}
    \begin{split}
        \int_{\bbR^d\times\bbR^d}\frac{F(x_1,x_2)}{|x_1-x_2|^{\gamma}}dx_1dx_2&=2^{\frac{\gamma}{2}}\int_{\bbR^d}\frac{\tilde{f}(y)}{|y|^{\gamma}}\,dy\\
        &\le2^{\frac{\gamma}{2}} \left( \int_{|y|>1}\tilde{f}(y)dy+\int_{|y|\le1}\frac{\tilde{f}(y)}{|y|^{\gamma}}dy \right).
    \end{split}
\end{equation}\label{fisherIII.2}
The first term does not exceed 1, while for the second term one applies H\"older's inequality and \eqref{fisher:lp},
\begin{equation}
    \int_{|y|\le1}\frac{\tilde{f}(y)}{|y|^{\gamma}}dy\le \left(\int_{|y|\le1}|y|^{-\frac{\gamma}{\beta}}dy\right)^{\beta}
    \|\tilde{f}\|_{\frac{1}{1-\beta}}
    \le C_{\gamma,\beta}I(\tilde{f})^{\frac{\beta d}{2}}\le 2C_{\gamma,\beta}I_2(F)^{\frac{d\beta }{2}} .
\end{equation}

Note that the restriction $\frac{\gamma}{\beta}< d$ comes from the integrability of $|y|^s$ while $\beta\le\frac{2}{d}$ comes from \eqref{fisher:lp}. Therefore \eqref{fisher:exp1} holds. For \eqref{fisher:exp2}, one has
\begin{equation}
\begin{split}
    &\int_{|x-y|<\epsilon}\frac{F(x,y)}{|x-y|^{\gamma}}dxdy
    =\int_{|y|\le\frac{\epsilon}{\sqrt{2}}}\frac{\tilde{f}(y)}{|y|^{\gamma}}dy\\
    &\le \left(\int_{|y|\le\frac{\epsilon}{\sqrt{2}}}|y|^{-\frac{\gamma}{\beta}}dy \right)^{\beta}\|\tilde{f}\|_{\frac{1}{1-\beta}}
    \le C_{\gamma,\beta}\epsilon^{d\beta-{\gamma}}I(\tilde{f})^{\frac{d\beta }{2}}\le 2C_{\gamma,\beta}\epsilon^{d\beta-{\gamma}}I_2(F)^{\frac{d\beta}{2}},
    \end{split}
\end{equation}
which implies \eqref{fisher:exp2}.

2. For $d\ge 3$ and $\gamma=2$, we note
\begin{multline*}
\int_{\bbR^d}\frac{1}{|y|^2}\tilde{f}(y)\,dy
=-\frac{1}{d-2}\int_{\bbR^d}\frac{y}{|y|^2}\cdot\nabla\tilde{f}(y)\,dy \\
\le \frac{1}{d-2}\left(\delta\int_{\bbR^d} \frac{1}{|y|^2}\tilde{f}(y)\,dy+\frac{1}{4\delta}\int_{\bbR^d}\frac{|\nabla\tilde{f}(y)|^2}{\tilde{f}(y)}\,dy\right).
\end{multline*}
One can choose $\delta=\frac{d-2}{2}$ and obtain
\[
\int_{\bbR^d}\frac{1}{|y|^2}\tilde{f}(y)\,dy\le \frac{1}{(d-2)^2}I(\tilde{f}).
\]
The integration by parts can be easily justified by approximating $\tilde{f}$ with compactly supported smooth functions. The claim therefore follows.
\end{proof}

\subsection{The $N$-particle system}
In this part, we study the $N$-particle system \eqref{model} and provide some estimates on the entropy and energy. Most of the results have been established in \cite{liuyang2016}, but we will give alternative proofs here for the convenience of the readers. These results will be used further  in the proof for propagation of chaos result in section \ref{sec:3dchaos}.

Assume the dimension $d\ge3$ and throughout this part and $N$ is set to be fixed. We will also use $X_t^i$ to represent $X_t^{i,N}$ in this section for convenience. The joint distribution of the particles $(X_t^1,...,X_t^N)$ is denoted by $f_t^N$.  The important quantities associated with the system include entropy and energy.
Recall the entropy defined in \eqref{eq:normalizedentropy} and we also define the energy given by (recall \eqref{eq:empirical} for the empirical measure $\mu^N$)
\begin{gather}
\mathscr{E}_N(t):=\frac{1}{2}\iint_{\mathcal{D}^c}g(x-y)\mu^N(dx)\mu^N(dy)(t)=\frac{1}{2N^2}\sum_{i\neq j}g(X_t^i-X_t^j),
\end{gather}
where $\mathcal{D}$ represents the diagonal $\{(x,y): x=y\}$.  For the convenience, we define
\begin{gather}
h_0=(-\Delta)^{-1}\rho_0=\int_{\bbR^d} g(x-y)\rho_0(y)\,dy.
\end{gather}
The continuous system has an initial energy:
\begin{multline}
\mathcal{E}(\rho_0):=\frac{1}{2}\iint_{(\bbR^{d})^2} g(x-y)\rho_0(x)\rho_0(y) dx dy \\
=\int_{\bbR^d} |\nabla h_0(x)|^2\,dx=\|\rho_0\|_{H^{-1}}^2
\le C\|\rho_0\|_{\frac{2d}{d+2}}^2,
\end{multline}
by Hardy-Littlewood-Sobolev inequality. Moreover, it holds that
\[
\mathbb{E}\mathscr{E}_N(0)=\frac{N-1}{N}\mathcal{E}(\rho_0).
\]

We first of all state the results about the well-posedness of the system \eqref{model}:
\begin{theorem}\label{noncollision}
For any $d\ge3$ and $N\ge2$, consider a sequence of independent $d$-dimensional Brownian motions $\{(B_t^i)_{t\ge0}\}_{i=1}^N$ and the independent and identically distributed (i.i.d.) initial data $\{X_0^i\}_{i=1}^N$ with a common distribution $f_0$ satisfying $H_1(f_0)<+\infty$ and a common density $\rho_0\in L^{\frac{2d}{d+2}}(\bbR^d)\cap L^1(\bbR^d,(1+|x|^2)dx)$. Then there exists a unique global strong solution to \eqref{model} with $X_t^i\neq X_t^j$ a.s. for all $t>0$ and $i\neq j$.
\end{theorem}

The proof for the non-collision result and energy estimate is based on mollification approximation. Recall that the potential $g(x)=C_d|x|^{2-d}$ is the solution to $-\Delta g=\delta$, and the mollification we use is given by 
\begin{gather}
g_{\e}=J_\e*g,~~~ F_\e(x)=-\nabla g_\e(x),
\end{gather}
where $J_{\e}=\frac{1}{\e^d}J(\frac{x}{\e})$, for some fixed $J(x)\in C^2(\bbR^d)$ which is non-negative, radial, with $\supp J(x)\subset B(0,1)$ and $\int_{\bbR^d} J(x)\,dx=1$. This mollification has the following standard properties, for which we omit the proofs.
\begin{lemma}[\cite{liuyang2016}, Lemma 2.1]\label{mollification}
\begin{enumerate}[(i)]
  \item $g_\e(x)=g(x), F_\e(x)=F(x)$ whenever $|x|\ge\e$;
  \item $F_\e(0)=0$, $\nabla\cdot F_\e(x)=J_\e(x)$;
  \item $|F_\e(x)|\le\min\{\frac{C_d|x|}{\e^d},|F(x)|\}$.
\end{enumerate}
\end{lemma}

We first consider the following regularized system for \eqref{model}:
\begin{equation}\label{regmodel}
    dX_t^{i,\e}=\frac{1}{N}\sum_{j=1,j\neq i}^NF_\e(X_t^{i,\e}-X_t^{j,\e})dt+\sqrt{2}dB_t^i, \ \ X_t^{i,\e}|_{t=0}=X_0^i.
\end{equation}
We try to use system \eqref{regmodel} to approximate \eqref{model}. Since $F_\e\in C_b^2(\bbR^d)$, \eqref{regmodel} is well-defined and has a unique strong solution.
We start with the a priori estimates of the entropy and energy for this regularized system.
\begin{lemma}\label{u-estimateI}
Let  $\{X_t^{i,\e}\}_{i=1}^N$ be the unique strong solution to \eqref{regmodel} with joint distribution $(f_t^{N,\e})_{t\ge0}$ and density $(\rho_t^{N,\e})_{t\ge0}$. Then we have the following relation for energy
\begin{multline}\label{energy:est1}
\langle \rho_t^{N,\e}, E^{N,\e}\rangle+\int_0^t\langle \rho_s^{N,\e}, |F_1^{N,\e}|^2\rangle ds
+\frac{N-1}{N}\int_0^t\langle \rho_s^{N,\e}, J^{\e}(x_1-x_2)\rangle\,ds \\
    = \frac{N-1}{2N}\iint_{(\bbR^d)^2} g^{\e}(x-y)\rho_0(x)\rho_0(y)\,dxdy\le C_1\mathcal{E}(\rho_0),
\end{multline}
and uniform estimates for entropy and second moment as follows
\begin{equation}\label{enmo:est1}
\begin{split}
&H_N(f_t^{N,\e})+\int_0^tI_N(f_s^{N,\e})ds+\frac{N-1}{N}\int_0^t\langle \rho_s^{N,\e}, J^{\e}(x_1-x_2)\rangle\,ds =H_N(f_0^N)=H_1(\rho_0),\\
&\bbE[|X_t^{i,\e}|^2]\le 3\bbE[|X_0^1|^2]+Ct\mathcal{E}(\rho_0)+6td.
\end{split}
\end{equation}
Here
\begin{gather}
E^{N,\e}(x)=\frac{1}{2N^2}\sum_{i,j=1,i\neq j}^Ng^{\e}(x_i-x_j),~~
F_1^{N,\e}(x)=\frac{1}{N}\sum_{j=2}^N F_{\e}(x_j-x_1).
\end{gather}
\end{lemma}

\begin{proof}[Sketch of the proof]
Since the force field is bounded and smooth and the initial density $\rho_0^{N,\e}$ is continuous, $\rho_t^{N,\e}$ is a classical nonnegative solution to the Fokker-Planck equation
\begin{equation}\label{eq:Ndensitymollified}
    \partial_t\rho_t^{N,\e}=\frac{1}{2}\nabla\cdot(\rho_t^{N,\e}\nabla g^{N,\e})+\Delta\rho_t^{N,\e},
\end{equation}
where
\[
g^{N,\e}(x)=\frac{1}{N}\sum_{i,j: i\neq j} g^{\e}(x_i-x_j) \Rightarrow \nabla_{x_1}g^{N,\e}(x)=-2F_1^{N,\e}.
\]

For \eqref{energy:est1}, one starts with \eqref{eq:Ndensitymollified} to obtain
\begin{equation}\label{u-estimate1}
    \frac{d}{dt}\langle \rho_t^{N,\e}, g^{N,\e}\rangle=- \left\langle \rho_t^{N,\e},\frac{1}{2}|\nabla g^{N,\e}|^2 \right\rangle-\left\langle \rho_t^{N,\e},\frac{2}{N}\sum_{i,j=1,i\neq j}^NJ_\e(x_i-x_j) \right\rangle.
\end{equation}
By exchangeability,
\[
\langle \rho^{N,\e}, |\nabla g^{N,\e}|^2\rangle=
N\langle \rho^{N,\e}, |\nabla_{x_1}g^{N,\e}|^2\rangle
=N\langle \rho^{N,\e}, |F_1^{N,\e}|^2\rangle.
\]
By dividing both sides by $2N$ and integrating on time, the equality in \eqref{energy:est1} follows. Moreover,
\begin{multline*}\label{eq:mollifiedinitial}
\frac{1}{2N}\langle \rho_0^{N,\e}, g^{N,\e}\rangle
=\frac{N-1}{2N}\iint_{(\bbR^d)^2} g^{\e}(x-y)\rho_0(x)\rho_0(y)\,dxdy \\
\le \frac{1}{2}\int_{\bbR^d} \nabla h_0(x)\cdot\nabla h_0^{\e}(x)\,dx
\le C\|\nabla h_0\|_2^2.
\end{multline*}
This holds because $\|\nabla h_0^{\e}\|_2\le \|\nabla h_0\|_2$ by Young's convolutional inequality.

Now by simple computations and integration by parts,
\begin{equation}
    \frac{d}{dt}H_N(f_t^{N,\e})=-I_N(f_t^{N,\e})-\frac{1}{N^2}\int_{\bbR^{Nd}}\sum_{i,j=1,i\neq j}^NJ_\e(x_i-x_j)\rho_t^{N,\e}dx,
\end{equation}
which gives the entropy relation in \eqref{enmo:est1} since $J_\e$ is non-negative.

For the moment estimate in \eqref{enmo:est1}, since $\bfX_t^{N,\e}$ is the solution to \eqref{regmodel}, one can deduce that
\begin{equation}\label{u-estimate2}
    |X_t^{i,\e}|^2\le 3|X_0^{i}|^2+\frac{3t}{N^2}\int_0^T\left|\sum_{j=1,j\neq i}^NF_\e(X_s^{i,\e}-X_s^{j,\e})\right|^2ds+6|B_t^i|^2 .
\end{equation}
Taking expectation of \eqref{u-estimate2}, and noting exchangeability, one has
\begin{align*}
\bbE\left[ \Big|\sum_{j=1,j\neq i}^NF_\e(X_s^{i,\e}-X_s^{j,\e}) \Big|^2\right] &=\bbE\left[\frac{1}{N}\sum_{i=1}^N
\Big|\sum_{j=1,j\neq i}^NF_\e(X_s^{i,\e}-X_s^{j,\e})\Big|^2\right]\\
&=N^2\langle \rho_s^{N,\e}, |F_1^{N,\e}|^2\rangle.
\end{align*}
Then, the moment estimate in \eqref{enmo:est1} follows from \eqref{energy:est1} directly.
\end{proof}
\begin{proof}[Proof for Theorem \ref{noncollision}]
First we restrict ourselves to a finite time period $[0,T]$. In order to show that the particles in \eqref{model} a.s. never collide, we consider system \eqref{regmodel} along with the stopping time
$$\tau_{\epsilon}=\inf\{t\ge0|\min_{i\neq j}|X_t^{i,\e}-X_t^{j,\e}|\le\e\}.$$
Since \eqref{model} and \eqref{regmodel} takes the same initial value, and by the fact that $F_\e(x)=F(x)$ whenever $|x|\ge\e$, for any $\e_1>\e$ and $\omega\in\Omega$, if we define
\begin{equation}\label{noncollision1}
    \hat{X_s}^\e(\omega)=\textbf{1}_{\tau_{\e_1}(\omega)\le t}X_s^\e(\omega)+\textbf{1}_{\tau_{\e_1}(\omega)> t}X_s^{\e_1}(\omega)
\end{equation}
then \eqref{noncollision1} is also a solution for system \eqref{regmodel} on $[0, t]$ since $F_\e(x)=F_{\e_1}(x)$ when $|x|\ge\e_1$. Therefore from the uniqueness of the solution (since $F_\e$ is Lipschitz over $\bbR^d$), we see that
\begin{equation}\label{noncollision2}
   \bbP(X_s^{\e_1}\textbf{1}_{\tau_{\e_1}> t}=X_s^\e\textbf{1}_{\tau_{\e_1}> t},\forall0\le s\le t)=1.
\end{equation}
Now we consider the set
\[
A_{\e, \e_1}:=\bigcap_{t\in\mathbb{Q}}\{X_s^{\e_1}\textbf{1}_{\tau_{\e_1}> t}=X_s^\e\textbf{1}_{\tau_{\e_1}> t},\forall0\le s\le t \},
\]
where $\mathbb{Q}$ is the set of rational numbers.
By \eqref{noncollision2} $\bbP(A_\e)=1$. For $\omega\in A_\e$, if $\tau_\e(\omega)<\tau_{\e_1}(\omega)$, then there exists a rational number $t\in\mathbb{Q}$ such that $\tau_\e(\omega)<t<\tau_{\e_1}(\omega)$, then by the definition of $A$ we see that $X_s^{\e_1}=X_s^\e$ for $0\le s\le t$, which contradicts with the assumption $\tau_\e(\omega)<t$. Therefore we have proved that when $\e_1>\e$, $\tau_\e\ge\tau_{\e_1}$ for a.s. $\oino$.

 We take $\e_n=\frac{1}{2^n}$. Consider
 \[
 A=\bigcap_{n\ge m\ge 1}A_{\epsilon_n, \epsilon_m}.
 \]
 From the discussion above, we see that $\{\tau_{\e_n}\}$ is non-decreasing when $n\to\infty$ for $\omega\in A$ and $\bbP(A)=1$.  
 
 Define
 \[
 A_0:=\left\{\tau_{\e_n}\uparrow+\infty \right\}.
 \]
 If we can show that
 \begin{equation}\label{noncollision3}
    \bbP(A_0)=1,
 \end{equation}
 then for $\omega\in A_0\cap A$ there exists an $M(\omega)$ such that $\tau_{\e_n}(\omega)>T$ when $n\ge M(\omega)$. This implies that $X_t^{\e_n}(\omega)=X_t^{\e_{M(\omega)}}(\omega)$ for $n\ge M(\omega)$ and $0\le t\le T$. Therefore we can define
\begin{equation}
    \tilde{X}_t(\omega)=X_t^{\e_{M(\omega)}}(\omega)
\end{equation}
whenever $\omega\in A_0\cap A$, and for other $\omega\in\Omega$ we just put $\tilde{X}_t(\omega)=X_0(\omega)$. Then $\tilde{X}_t$ satisfies \eqref{model} when $0\le  t\le T$ a.s. $\omega\in\Omega$,  which gives the existence of the solution.

For the uniqueness, suppose that $X_t$ is another solution that solves \eqref{model}. Consider the stopping time
$$\sigma_{\epsilon}=\inf\{t\ge0|\min_{i\neq j}|X_t^{i}-X_t^{j}|\le\e\}.$$
Similar to \eqref{noncollision1},
\begin{equation}
    \hat{X_s}^\e(\omega)=\textbf{1}_{\sigma_{\e}(\omega)\le t}X_s^\e(\omega)+\textbf{1}_{\sigma_{\e}(\omega)> t}X_s(\omega)
\end{equation}
gives a solution for \eqref{regmodel}, from which by using the uniqueness it is not hard to see the set
\[
A_1:=\bigcap_{t\in \mathbb{Q}}\{X_s\textbf{1}_{\sigma_{\e_n}> t}=X_s^\e\textbf{1}_{\sigma_{\e_n}> t},\forall0\le s\le t, n\ge1\}
\]
satisfies $\bbP(A_1)=1$.
Now for $\omega\in A_1$, if $\sigma_{\e}<\tau_{\e}$ for some $\e=\e_n$, since for fixed $\omega$, $X_t$ and $X_t^{\e}$ are continuous in $t$, from the definition of the stopping time we see
$$\min_{i\neq j}|X_{\sigma_{\e}}^i-X_{\sigma_{\e}}^j|=\e, \min_{i\neq j}|X_{\sigma_{\e}}^{i,\e}-X_{\sigma_{\e}}^j|>\e,$$
by continuity there exists a $t\in\mathbb{Q}$ such that
$$\min_{i\neq j}|X_t^i-X_t^j|-\min_{i\neq j}|X_t^{i,\e}-X_t^{j,\e}|<0,\ \ t<\sigma_\e,$$
which contradicts with the definition of $A_1$. This gives the fact that $\sigma_{\e_n}(\omega)\ge\tau_{\e_n}(\omega)$ as long as $\omega\in A_1$. Now if \eqref{noncollision3} holds then $\bbP(A_0\cap A\cap A_1)=1$, and for $\omega\in A_0\cap A\cap A_1$, we have  $X_t(\omega)=X_t^{\e_{M(\omega)}}(\omega)=\tilde{X}_t(\omega)$, which concludes the proof for uniqueness.

Now we show \eqref{noncollision3}. Since $\tau_{\e_n}$ is a.s. non-decreasing, to show $\tau_{\e_n}\uparrow+\infty$, a.s., it suffices to show that for any fixed $T$,
\begin{equation}\label{noncollision4}
    \lim_{\e\to0}\bbP(\tau_\e\le T)=0.
\end{equation}
We consider the unnormalized energy 
\[
\Phi_t^{\e,N}:=\frac{1}{N}\sum_{i,j=1,i\neq j}^Ng^\e(X_t^{i,\e}-X_t^{j,\e}).
\]
 Then we have the following basic fact
\begin{equation}\label{noncollision5}
    \{\tau_\e\le T\}\subset
    \left\{\sup_{t\in[0,T]}\Phi_{t\wedge\tau_\e}^{\e,N}
    \ge\Phi_{\tau_\e}^{\e,N}\right\}.
\end{equation}
Since $g^\e\in C_b^2(\bbR^d)$, by It\^{o}'s formula and the fact that $-\Delta g^\e(x)=J_\e(x)=0$ on $|x|\ge\e$, we get
\begin{multline}\label{noncollision6}
\Phi_t^{\e,N}=\Phi_0^{\e,N}-\frac{2}{N^2}\int_0^{t}\sum_{i=1}^N\left|\sum_{j=1,j\neq i}^NF^\e(X_s^{i,\e}-X_s^{j,\e}) \right|^2ds\\
    -\frac{2}{N}\sum_{i,j=1,i\neq j}^N\int_0^tJ_{\e}(X_s^{i,\e}-X_s^{j,\e})ds-M_{t}^\e,
\end{multline}
where
\[
M_t^\e=\frac{\sqrt{2}}{N}\sum_{i,j=1,i\neq j}^N\int_0^{t}F_\e(X_s^{i,\e}-X_s^{j,\e})\cdot(dB_s^i-dB_s^j)=\frac{2\sqrt{2}}{N}\sum_{i,j=1,i\neq j}^N\int_0^{t}F_\e(X_s^{i,\e}-X_s^{j,\e})\cdot dB_s^i.
\]
Since $F_\e$ is bounded, $M_t^\e$ is a martingale with
\begin{equation}\label{noncollision7}
\begin{split}
    \bbE[(M_t^{\e})^2]&=\frac{8}{N^2}\int_0^t\sum_{i=1}^N\bbE\left[\Big|(\sum_{j=1,j\neq i}^NF_\e(X_s^{i,\e}-X_s^{j,\e})) \Big|^2 \right]ds\\
    &=\int_0^t\langle \rho_s^{N,\e},2|\nabla g^{N,\e}|^2\rangle ds\\
    &=2N\int_0^t \langle \rho^{N,\e}, |F_1^{N,\e}|^2\rangle\,ds  \le CN\mathcal{E}(\rho_0).\\
\end{split}
\end{equation}
The last inequality comes from \eqref{energy:est1}. Combining \eqref{noncollision5} and \eqref{noncollision6}, from the positivity of $g^\e$ (since $d\ge3$) we have
\begin{align}\label{noncollision8}
    \{\tau_\e\le T\} &\subset \left\{(\Phi_0^{\e,N}-\inf_{0\le t\le T}M_t^\e)>\frac{1}{N}g(\e) \right\} \\
     &\subset\{\Phi_0^{\e,N}>R\}\cup \left\{\sup_{0\le t\le T}(-M_t^\e)>\frac{1}{N}g(\e)-R \right\}.
\end{align}
for any $R>0$. Here $g(\e)=C_d\e^{2-d}$. We notice that
\begin{equation}
    \bbE[\Phi_0^{\e,N}]=(N-1)\int_{\bbR^{2d}}\rho_0(x)\rho_0(y)g^\e(x-y)dxdy
    \le C(N-1)\mathcal{E}(\rho_0).
\end{equation}
Therefore, Markov's inequality gives
\begin{equation}\label{noncollision9}
\bbP(\Phi_0^{\e,N}>R)\le\frac{1}{R}(N-1)C\mathcal{E}(\rho_0).
\end{equation}
 For the second term, we apply Doob's inequality for martingales (p.203, Theorem 7.31 in \cite{klebaner2012})
\begin{equation}\label{noncollision10}
    \begin{split}
        \bbP\left(\sup_{0\le t\le T}(-M_t^\e)>\frac{1}{N}g(\e)-R\right)  &\le
        \frac{N}{g(\e)-NR} \left( \bbE[|\sup_{0\le t\le T}(-M_t^\e)|^2] \right)^{\frac{1}{2}}\\
        &\le \frac{4N}{g(\e)-NR}(\bbE[|M_T^\e|^2])^{\frac{1}{2}}\le\frac{C(N,d,\mathcal{E}(\rho_0))}{g(\e)-NR},
    \end{split}
\end{equation}
where we used \eqref{noncollision7}. Combining \eqref{noncollision8}, \eqref{noncollision9} and \eqref{noncollision10}
\begin{equation}
    \bbP(\tau_\e\le T)\le C(N)\left(\frac{1}{R}+\frac{1}{g(\e)-NR} \right).
\end{equation}
We take $R=g(\e)^{\frac{1}{2}}$ and the conclusion follows from the fact that $g(\e)=C(d)\e^{2-d}\to\infty$ as $\e\to0$.

Lastly, we conclude the global existence and uniqueness. For $k\ge1$, suppose $X_t^{(k)}$ is the a.s. unique solution to \eqref{model} on the time interval $t\in[0,k]$. From the previous local existence and uniqueness proof, we find that the set
\[
S_0=\bigcap_{\ell\ge 1}\Big\{X_t^{(k)}=X_t^{(\ell)},\ \forall k\ge \ell,t\in[0,\ell] \Big\}
\]
has probability 1. Therefore if we define
$$\tilde{X}_t(\omega)=X_t^{([t+1])}(\omega)\textbf{1}_{S_0}(\omega)+X_0(\omega)\textbf{1}_{S_0^c}(\omega)$$
then $\tilde{X}_t$ satisfies \eqref{model} for all $t>0$ a.s.. (Here $[t]$ rounds $t$ to the nearest integer). Meanwhile, if $X_t$ is another global solution for $t>0$, then by local uniqueness we know that for any $k\ge1$, $X_t=X_t^{(k)},\forall 0\le t\le k$, a.s.. This implies that $X_t=\tilde{X}_t$, which gives the global uniqueness.
\end{proof}

Next we state some useful estimates for the $N$-particle system \eqref{model}.
\begin{proposition}\label{u-estimateII}
Suppose $\{X_t\}_{t\ge0}$ is the solution for \eqref{model} with joint distribution $f_t^N$. Then $f_t^N$ has a density function $\rho_t^N$. Moreover, we have the following estimates:
\begin{equation}\label{entropy:est2}
    H_N(f_t^{N})+\int_0^tI_N(f_s^{N})ds\le H_1(f_0)(=H_N(f_0^N)),
\end{equation}
\begin{equation}\label{energy:est2}
    \mathbb{E}\mathscr{E}_N(t)+\mathbb{E}\int_0^t\frac{1}{N}\sum_{i}\Big|\frac{1}{N}\sum_{j: j\neq i}F(X_j-X_i) \Big|^2 ds\le \frac{N-1}{N}\mathcal{E}(\rho_0),
\end{equation}
\begin{equation}\label{moment:est2}
    \bbE[|X_t^{i}|^2]\le 3\bbE[|X_0^1|^2]+Ct \mathcal{E}(\rho_0)+6td.
\end{equation}
\end{proposition}

\begin{proof}
Throughout this proof, $C$ denotes the constant which depends on $N,\rho_0,d$ and so on, but not on $\e$. Note that for any density $\rho\in L^1(\bbR^{Nd})$, $|\rho\log^-\rho|<1$. For $\alpha<1$,
\begin{equation}
    \int_{|x|\ge1}|\rho(x)\log^-\rho(x)|\,dx\le C \int_{|x|\ge1}|\rho|^\alpha dx\le C\left(\int_{\bbR^{Nd}}|x|^2\rho dx\right)^{\alpha}\Big(\int_{|x|\ge1}|x|^{-\frac{2\alpha}{1-\alpha}}dx\Big)^{1-\alpha}.
\end{equation}
If we take $\frac{Nd}{Nd+2}<\alpha<1$, then from the uniform estimate \eqref{enmo:est1} we deduce
\begin{equation}\label{eq:L1logL1}
    \int_{\bbR^{Nd}}|\rho_t^{N,\e}(x)\log\rho_t^{N,\e}(x)|dx \le C ,
\end{equation}
which means that $\rho_t^{N,\e}$ is uniformly integrable in $L^1(\bbR^{Nd})$. 
Consider the sequence $\{\rho_t^{N,\e_n}: \e_n=\frac{1}{2^n}\}$. Denote by $B_r=\{|x|\le r\}$ the ball in $\bbR^{Nd}$ centered at the origin with radius $r$.
By Dunford-Pettis theorem (p.412, Theorem 12 in \cite{royden2010}), there exists a subsequence $\{\rho_{t,(1)}^{N,\e_n}\}$ converging weakly in $L^1(B_1)$ to some $\rho^N_{t,(1)}$.  This subsequence has a further subsequence converging weakly in $L^1(B_2)$ to some $\rho^N_{t,(2)}$. From the uniqueness of the weak limit we see that $\rho^N_{t,(2)}=\rho^N_{t,(1)}$ a.e. on $B_1$. Proceeding this process and taking the diagonal sequence, there exists a subsequence (without relabeling) and a $\rho_t^N\in L^1_{\loc}(\bbR^{Nd})$ such that
\begin{equation}\label{u-estimateII.1}
    \rho_t^{N,\e_n}\rightharpoonup\rho_t^N~\text{in}~ L^1(B_k), \forall k\ge1.
\end{equation}
\eqref{u-estimateII.1} also gives $\|\rho_t^N\|_{L^1(\bbR^{Nd})}\le1$. Now the moment estimate in \eqref{enmo:est1} gives the tightness of $\rho_t^{N,\e}$, i.e., $\int_{|x|\ge M}\rho^{N,\e}(x)dx$ goes to 0 as $M\to\infty$ uniformly in $\e$, by \eqref{u-estimateII.1} it is not hard to observe that for any $\phi\in L^\infty(\bbR^{Nd})$, $\langle \phi,\rho_t^{N,\e_n}\rangle\to\langle\phi,\rho_t^{N}\rangle$, i.e.,
\begin{equation}
    \rho_t^{N,\e_n}\rightharpoonup\rho_t^N~\text{in}~ L^1(\bbR^{Nd}).
\end{equation}
In particular, $\|\rho_t^N\|_{L^1(\bbR^{Nd})}=1$.
From the proof of Theorem \ref{noncollision} we see that $X_t^{\e_n}\to X_t$, a.s., therefore for any $\varphi\in C_b(\bbR^{Nd})$, we have
\begin{equation}
    \langle \varphi,\rho_t^N\rangle=\lim_{n\to\infty}\langle\varphi,\rho_t^{N,\e_n}\rangle=\lim_{n\to\infty}\bbE[\varphi(X_t^{\e_n})]=\bbE[\varphi(X_t)]=\langle \varphi,f_t^N\rangle
\end{equation}
This gives the fact that $f_t^N$ has density $\rho_t^N$. Note that
\begin{equation}
\langle \rho_t^{N,\e}, E^{N,\e}\rangle+\int_0^t\langle \rho_s^{N,\e}, |F_1^{N,\e}|^2\rangle ds
=\bbE[E^{N,\e}(\bfX_t^{N,\e})]+\int_0^t\bbE[|F_1^{N,\e}(\bfX_s^{N,\e})|^2]ds
\end{equation}
and that for a.s. $\oino$, $\bfX_t^{N,\e_n}(\omega)=\bfX_t^{N}(\omega)$ when $n$ is big enough, from the Fatou's lemma, the exchangeability, and \eqref{energy:est1} we obtain \eqref{energy:est2}. Similarly, \eqref{moment:est2} holds. Now combining the entropy estimate in \eqref{enmo:est1} and the fact that the functionals $H$ and $I$ are both lower semicontinuous with respect to weak convergence (Theorem 5.4 and Theorem 5.7 in \cite{hauray2014}), we see that
\begin{equation}
     H_N(f_t^{N})+\int_0^tI_N(f_s^{N})ds\le\liminf_{n\to\infty} \left( H_N(f_t^{N,\e_n})+\int_0^tI_N(f_s^{N,\e_n})ds \right) \le H_1(\rho_0),
\end{equation}
which gives \eqref{entropy:est2}.

\end{proof}

\subsection{The weak convergence of the empirical measures}
In this part we recall the results in \cite{liuyang2016} for the weak convergence of the empirical measures.
\begin{proposition}[\cite{liuyang2016}, Lemma 3.1]\label{tightI}
For any $N\ge2$ and $d\ge3$, let $\{(X_t^{i,N})_{0\le t\le T}\}_{i=1}^N$ be the unique strong solution to \eqref{model} with the i.i.d. initial data $\{X_0^{i,N}\}_{i=1}^N$. Suppose the common density $\rho_0(x)\in L^{\frac{2d}{d+2}}(\bbR^d)\cap L^1(\bbR^d,(1+|x|^2)dx)$ and $H_1(\rho_0)<\infty$. Recall the empirical measure $\mu^N$ defined in \eqref{eq:empirical}. Then

(i) The sequence $\{\mathcal{L}(X^{1,N})\}$ is tight in $\bfP(C([0,T];\bbR^d))$;

(ii) The sequence $\{\mathcal{L}(\mu^N)\}$ is tight in $\bfP(\bfP(C([0,T];\bbR^d)))$.
\end{proposition}
Here $\mathcal{L}(X^{1,N})$ is the law of $X^{1,N}$, i.e. $\mathcal{L}(X^{1,N})(A) = \bbP(X^{1,N}\in A)$ for $A\subset C([0,T];\bbR^d)$ Borel measurable; Similar for $\mathcal{L}(\mu^N)$.
For the convenience of the readers, we provide a concise proof in Appendix \ref{app:missingpf}.

We consider the projection $\pi_t: (C([0,T];\bbR^d), \mathcal{B})\to \bbR^d$ where $\mathcal{B}$ is the standard Borel $\sigma$-algebra in $C([0,T];\bbR^d)$:
\[
\pi_t(X)=X(t).
\]
Then, for some measure $\nu \in \bfP(C([0,T];\bbR^d))$,  we define the time marginal $\nu_t$ as the pushforward of $\nu$ under $\pi_t$:
\begin{gather}\label{eq:mut}
\nu_t:=(\pi_t)_{\#}\nu
\end{gather}
or
\[
\nu_t(E)=\nu(\pi_t^{-1}(E)),~\forall E\in \bbR^d, \text{Borel measurable.}
\]
Consequently, we have
\[
\Big( t\mapsto \nu_t \Big)\in C([0, T]; \bfP(\mathbb{R}^d)),
\]
where $\bfP(\mathbb{R}^d)$ is equipped with the topology of weak convergence.

We easily conclude the following by change of measures, i.e., for $T:(\mathcal{X}, \nu)\to (\mathcal{Y}, \tilde{\nu})$ with $\tilde{\nu}(A) = \nu(T^{-1}(A))$, one has $\int_\mathcal{Y} fd\tilde{\nu} = \int_\mathcal{X} f\circ Td\nu$. 
\begin{lemma}\label{substitution}
Suppose $\nu\in\mathbf{P}(C([0,T];\bbR^d))$ with time marginal $\nu_t\in\bfP(\bbR^d)$, and $\psi$ is a Borel measurable function on $\bbR^d$ . Then for $0\le t\le T$, the equation
\begin{equation}
    \int_{C([0,T];\bbR^d)}\psi(X_t)\nu(dX)=\int_{\bbR^d}\psi(x)\nu_t(dx)
\end{equation}
holds if either side is integrable. Similarly,
for the product space $C([0,T];\bbR^d)\times C([0,T];\bbR^d)$ and Borel measurable function $\psi$ on $\bbR^d\times\bbR^d$,
\begin{multline}\label{substitutionII}
\iint_{C([0,T];\bbR^d)^2}\psi(X_t, Y_t)\nu(dX)\nu(dY)=\int_{C([0,T];\bbR^d)}\int_{\bbR^d}\psi(x,Y_t)\nu_t(dx)\nu(dY)\\
=\iint_{\bbR^d\times\bbR^d}\psi(x,y)\nu_t(dx)\nu_t(dy)
\end{multline}
if either side of \eqref{substitutionII} is integrable.
\end{lemma}

Recall that a sequence of random variables $Z_n$ taking values in some Polish space $\mathcal{X}$ converges in law to $Z$ means that $\mathbb{E}\varphi(Z_n)\to \mathbb{E}\varphi(Z)$ for any $\varphi\in C_b(\mathcal{X})$ (i.e. bounded continuous functions). The following lemma gives the consequence of the tightness in Proposition \ref{tightI}.
\begin{lemma}\label{lmm:convergesinlaw}
\begin{enumerate}
\item There is a subsequence of the empirical measures, $\mu^N\in\bfP(C([0,T];\bbR^d))$ (without relabeling), and a random measure $\mu:(\Omega,\mathcal{F},\bbP)\to\bfP(C([0,T];\bbR^d))$ such that
\begin{equation}
    \mu^N\to\mu\ {\rm in \ law\ as}\ N\to\infty.
\end{equation}
(Or $\mathcal{L}(\mu^N)$ converges weakly to $\mathcal{L}(\mu)$ in $\bfP(\bfP(C([0, T]; \bbR^d)))$.)

\item For the subsequence in 1, $\mu^N_t$, as $\bfP(\bbR^d)$ valued random measures, converge in law to $\mu_t$. In other words, $\mathcal{L}(\mu_t^N)$ converges weakly to $\mathcal{L}(\mu_t)$ in $\bfP(\bfP(\bbR^d))$.
\end{enumerate}
\end{lemma}

\begin{proof}
The first claim follows from the tightness of $\{\mathcal{L}(\mu^N)\}$ in $\bfP(\bfP(C([0,T];\bbR^d)))$ by the Prokhorov's theorem.

For the second, we first note that a sequence $\nu^n\in \bfP(C([0,T];\bbR^d))$ converges weakly to
$\nu\in \bfP(C([0,T];\bbR^d))$ will imply that $(\nu^n)_t\in \bfP(\bbR^d)$ converges weakly to $\nu_t\in \bfP(\bbR^d)$. In fact, for any function $\phi\in C_b(\bbR^d)$, we have
$\int_{\bbR^d}\phi(x)d(\nu_n)_t=\int_{C([0, T];\bbR^d)}\phi(X_t) d\nu_n$. Note that $X\to \phi(X_t)$ is a continuous functional on $C([0, T]; \bbR^d)$ and thus
\[
\lim_{n\to\infty}\int_{\bbR^d}\phi(x)d(\nu_n)_t=\int_{C([0, T];\bbR^d)}\phi(X_t) d\nu
=\int_{\bbR^d}\phi(x)d\nu_t.
\]
Now, consider a continuous functional $\Gamma: \bfP(\bbR^d)\to \bbR$. We define
$\Gamma_1: \bfP(C([0, T]; \bbR^d))\to \bbR$ as
\[
\Gamma_1(\nu):=\Gamma(\nu_t).
\]
According to what has been justified, $\Gamma_1$ is a continuous functional on $\bfP(C([0, T]; \bbR^d))$. Consequently,
\[
\mathbb{E}\Gamma_1(\mu^N)\to \mathbb{E}\Gamma_1(\mu) \Rightarrow
\mathbb{E}\Gamma(\mu_t^N)\to \mathbb{E}\Gamma(\mu_t).
\]
This then verifies the second claim.
\end{proof}

The following lemma gives another property which will be useful to us.
\begin{lemma}\label{convinlaw}
Let $\mathcal{X}$ be a Polish space. Suppose $\mu^N,\mu$ are random measures on $\mathcal{X}$ (i.e., $\bfP(\mathcal{X})$-valued random variables), such that $\mu^N$ converge to $\mu$ in law. For any $\psi\in C_b(\mathcal{X}\times\mathcal{X})$, if we define a functional $\calk_\psi:\bfP(\mathcal{X})\to\bbR$ with
\[
\calk_\psi(\nu)=\int_{\mathcal{X}^2}\psi(X,Y)\nu(dX)\nu(dY),
\]
 then $\calk_\psi(\mu^N)\to\calk_\psi(\mu)$ in law as $N\to\infty$.
\end{lemma}

\begin{proof}
We consider the metric on $\bfP(\mathcal{X})$ induced by weak convergence. By p.23, Theorem 2.8 in \cite{billingsley2013}, $\nu^N\rightharpoonup\nu$ in $\bfP(\mathcal{X})$ implies that $\calk_\psi(\nu^N)\to\calk_\psi(\nu)$, therefore for any $\phi\in C_b(\bbR)$, $\phi\circ\calk_\psi$ is a bounded continuous functional on $\bfP(\mathcal{X})$, then $$    \bbE[\phi(\calk_\psi(\mu^N))]\to\bbE[\phi(\calk_\psi(\mu))], \ N\to\infty,$$
which gives the last claim.
\end{proof}

We note the following facts regarding the marginal distributions (see \cite[Lemma 5.6]{hauray2014} \cite[Theorem 4.1]{fournier2014} and \cite[Lemma 3.2]{liuyang2016}). The results are modified for our purpose here and we sketch a quick proof for reference.
\begin{proposition}\label{tightII}
Under the assumption of Proposition \ref{tightI}, denote by $(f_t^N)_{t\ge0}$ the joint distribution of $\{(X_t^{i,N})_{t\ge0}\}_{i=1}^N$ and $f_t^{(j),N}$ be the $j$th marginal of $f_t^N$ for any $j\ge1$. Then

(i) For any $j$ that is a positive integer, we have
\begin{multline}\label{eq:uniformmarginal}
\sup_{t\in[0,T],N} \int_{\bbR^{dj}}|x|^2f_t^{(j),N}(dx)<\infty,
~\sup_{t\in[0,T],N} H_j(f_t^{(j),N})<\infty,~
\sup_N\int_0^T I_j(f_t^{(j),N})dt<\infty.
\end{multline}

(ii) $f_t^{(j),N}$ has a density $\rho_t^{(j),N}$. Consider $\rho^{(j), N}=(\rho_t^{(j), N})\in L^1([0, T]\times \mathbb{R}^{dj})$. It has a subsequence $\rho^{(j),N}$ (without relabeling) weakly converging to $\rho^{(j)}$ in $L^1([0, T]\times \bbR^{dj})$ as $N\to\infty$, and also for a.e. $t\in [0, T]$, $f_j^{(j),N}=\rho_t^{(j),N}\,dx$ converges weakly to $\rho_t^{(j)}\,dx$ as probability measures. Besides,
\begin{gather}\label{eq:limitmarginal}
\sup_{t\in [0, T]}\int_{\bbR^{dj}}|x|^2\rho_t^{(j)}dx<\infty,~~
\sup_{t\in [0, T]}\int_{\bbR^{dj}}\rho_t^{(j)}|\log \rho_t^{(j)}|\,dx <\infty,~~
\int_0^t I_j(\rho_t^{(j)})ds<\infty.
\end{gather}
Moreover, let $\mu$ be the limit (random) measure of any further subsequence of $\mu^N$, and let $\mu_t$ be the time marginal of $\mu$.  Then, for a.e. $t\in [0, T]$, it holds that
\[
\int_{\bbR^{dj}}\rho_t^{(j)}\varphi \,dx=\mathbb{E}(\langle \mu_t^{\otimes j}, \varphi\rangle ),
~~\forall \varphi\in C_b(\bbR^{dj}).
\]

(iii) The entropy and Fisher information of the limit random measure $\mu$ satisfy that for all $t\in [0, T]$:
\begin{gather}
\begin{split}
\mathbb{E}(H_1(\mu_t))=\sup_{j\ge 1}H_j(\rho_t^{(j)})\le \liminf_{N\to\infty} H_N(f_t^N),\\
\mathbb{E}(I_1(\mu_t))=\sup_{j\ge 1}I_j(\rho_t^{(j)})\le \liminf_{N\to\infty} I_N(f_t^N).
\end{split}
\end{gather}

(iv) We have the following estimates for the Fisher information
\begin{equation}\label{fisher:est2}
    \bbE\int_0^T I_1(\mu_t)dt<C.
\end{equation}
Consequently, for a.s. $\omega$, $\mu(\omega)$ has a density $(\rho_t(\omega))_{t\in [0, T]}$. At time $t=0$, $\rho_t(\omega)=\rho_0$ for a.s. $\omega$.

\end{proposition}

\begin{proof}

(i) The second moment estimate follows directly from equation \eqref{moment:est2}.
Equation \eqref{fishersym} and \eqref{entropy:est2} implies that
$\int_0^T I_j(f_t^{(j),N})\,dt \le \int_0^T I_N(f_t^{(j),N})\,dt
\le H_1(f_0)-H_N(f_t^N)$.
By the second moment estimates and Lemma \ref{entropy2} with $p=2,\lambda=1$, we see that
\begin{equation}\label{fisherIV.2}
    H_1(\rho_T^{(1),N})\ge-C_{p,\lambda}-M_2(\rho_T^{(1),N})\ge-C ,
\end{equation}
where $C$ depends only on $\rho_0,T$ and $d$. By Lemma \ref{entropy1}, we have $H_N(\rho_T^{N})\ge H_1(\rho_T^{(1),N}) \ge -C$. We thus have
\begin{gather}\label{eq:uniformfisher}
\sup_{N}\int_0^T I_j(f_t^{(j),N})\,dt \le \sup_N\int_0^T I_N(f_t^{N})\,dt\le C(j,T,d)<\infty.
\end{gather}

We note that $H_N(f_t^N)$ is uniformly bounded. Then, by \eqref{eq:entropyrelations}, we have
(note that entropy can be negative)
\begin{gather}\label{eq:aux1}
H_j(f_t^{(j),N})
 \le \left(1+\frac{N-mj}{mj}\right)H_N(f_t^N)-\frac{N-mj}{mj}H_{N-mj}\left(f_t^{(N-mj),N}\right),
\end{gather}
where $m$ is an integer chosen so that $N-mj\in [0, j)$. A simple application of \eqref{lbentropy} with second moment gives the uniform bound for $H_j(f_t^{(j),N})$.

(ii) By the uniform second moments estimate, $\rho^{(j), N}$ is tight in $L^1([0, T]\times \bbR^{dj})$. Moreover, we have $\int_{\bbR^{dj}} \rho^{(j),N}|\log \rho^{(j), N}|\,dx$ to be uniformly bounded by the uniform estimates of $H_j$ and a similar calculation for \eqref{eq:L1logL1}.
Hence, $\rho^{(j), N}$ is uniformly integrable on $[0, T]\times \bbR^{dj}$.  Though the Dunford-Pettis theorem is stated for finite measures, combined with the tightness, the uniform integrability implies that $\rho^{(j),N}$ is weakly compact in $L^1([0, T]\times \bbR^{dj})$. Hence, we can find a subsequence $\rho^{(j),N}$ converging weakly to $\rho^{(j)}$ in $L^1([0, T]\times \bbR^{dj})$.

The second moment mapping $\nu\mapsto \int_{\bbR^{dj}} |x|^2 \nu(dx)$ is lower-semicontinous with respect to the topology of weak convergence, which can be seen by approximating $|x|$ with $|x|\wedge m$. After taking $\sup$ in $t$, it is still lower semicontinuous. It has been proved in \cite[Lemma 4.2]{fournier2014} that $H_j$ and $I_j$ are lower semi-continuous. Taking supremum in $t$, or taking integral of nonnegative lower semicontinuous functionals still yield lower semicontinuous functionals. Hence, taking $N\to\infty$ in \eqref{eq:uniformmarginal}, we get the corresponding estimates for $\rho_t^{(j)}$. The second moment and entropy estimates then yield $\sup_{t\in [0, T]}\int_{\bbR^{dj}} \rho_t^{(j)}|\log \rho_t^{(j)}|\,dx<\infty$ similarly as we did in \eqref{eq:L1logL1}.

We now take $\phi\in C[0, T]$ and $\varphi\in C_b(\bbR^{dj})$. Then, $\Gamma: C([0, T];\bbR^d)^j\to \bbR$ defined by $(X^1,\ldots, X^j)\to \int_0^T\phi(t)\varphi(X^1_t,\ldots,X_t^j)\,dt$ is a bounded continuous functional. A slight generalization of Lemma \ref{convinlaw} with $\mathcal{X}=C([0, T];\bbR^d)$ shows that
\[
\mathbb{E}\langle (\mu^N)^{\otimes j}, \Gamma \rangle \to \mathbb{E}\langle \mu^{\otimes j}, \Gamma\rangle
=\int_0^T\phi(t)\mathbb{E}\langle \mu_t^{\otimes j}, \varphi\rangle\,dt,
\]
where the last term is obtained by Fubini and the definition of $\mu_t$.
Let 
\[
\nu_t :=\mathcal{L}(\mu_t)\in \bfP(\bfP(\bbR^d)).
\]
 Define
\[
\nu_t^j:=\int_{\bfP(\bbR^d)} g^{\otimes j} \nu_t(dg)\in \bfP(\bbR^{dj}).
\]
By this definition, we have for any $\varphi\in C_b(\bbR^{dj})$ that
\[
\langle \nu_t^j, \varphi\rangle=\int_{\bfP(\bbR^d)}\int_{\bbR^{dj}}\varphi(x)g^{\otimes j}(dx)\nu_t(dg)=\mathbb{E}(\langle \mu_t^{\otimes j}, \varphi\rangle).
\]
This means that $\mathbb{E}\langle (\mu^N)^{\otimes j}, \Gamma \rangle \to\int_0^T \phi(t)\langle \nu_t^j, \varphi\rangle\,dt$.

On the other hand, by definition and Fubini theorem,
\begin{multline}\label{eq:auxmarginal}
\mathbb{E}\langle (\mu^N)^{\otimes j}, \Gamma\rangle
=\int_0^T\phi(t)\mathbb{E}\left(\frac{1}{N^j}\sum\varphi(X_t^{i_1,N},\ldots, X_t^{i_j, N}) \right)dt=
\frac{N!}{N^j(N-j)!}\int_0^T\phi(t)\times\\
\mathbb{E}\varphi(X_t^{1,N},\ldots, X_t^{j,N})dt
+\frac{1}{N^j}\sum_{\text{some } i_{k}'s \text{ are equal}}\int_0^T\phi(t)\mathbb{E}\varphi(X_t^{i_1,N},\ldots, X_t^{i_j, N})dt.
\end{multline}
Simple estimate shows that the second term goes to zero as $N\to\infty$; while the first term converges to $\int_0^T\phi(t)\int_{\bbR^{dj}}\varphi \rho_t^{(j)}\,dx dt$ by the results we have just proved.

Since $\phi(t)$ is arbitrary, for a fixed $\varphi$, we have for a.e. $t$ that
\begin{gather}\label{eq:marginaldensity}
\int_{\bbR^{dj}}\varphi \rho_t^{(j)}\,dx =\langle \nu_t^j, \varphi\rangle.
\end{gather}
Moreover, since $C_c^{\infty}$ is separable, we know for a.e. $t$ and all $\varphi\in C_c$ that
\eqref{eq:marginaldensity} holds.
Using the uniform second moment bounds of $\rho_t^{(j)}$ and $\nu_t^j$ (the proof of second moment for $\nu_t^j$ can be obtained similarly as for $\rho_t^{(j)}$), we know that they are tight. We thus can pass from $C_c$ to $C_b$ for these $t$.  Hence, $\rho_t^{(j)}$ is in fact the density of $\nu_t^j$ for a.e. $t\in [0, T]$.

Now, a slight generalization of Lemma \ref{convinlaw} with $\mathcal{X}=\bbR^{dj}$ shows that for $\varphi\in C_b(\bbR^{dj})$ that
\[
\mathbb{E}\langle (\mu_t^N)^{\otimes j}, \varphi\rangle \to \mathbb{E}\langle (\mu_t)^{\otimes j}, \varphi\rangle=\langle \nu_t^j, \varphi\rangle,
\]
since $\mu_t^N$ converges in law to $\mu_t$ by Lemma \ref{lmm:convergesinlaw}.
Together with a similar computation of \eqref{eq:auxmarginal} shows that we in fact have
\[
\lim_{N\to\infty}\int_{\bbR^{dj}}\varphi f_t^{(j),N}(dx)=\langle \nu_t^j, \varphi\rangle.
\]
This in fact means $f_t^{(j),N}=\rho_t^{(j),N}\,dx\rightharpoonup \nu_t^j$ for all $t$. Thus, for a.e. $t$, $\rho_t^{(j),N}\,dx\rightharpoonup \rho_t^{(j)}\,dx$ as probability measures.

(iii) In \cite[Lemma 4.2]{fournier2014}, it is proved that the functional $I_j$ is convex, proper and lower semi-continuous. Then, \cite[Lemma 5.6]{hauray2014} showed that
\[
\bbE(I_1(\mu_t))=\int_{\bfP(\bbR^d)}I_1(g)\nu_t(dg)=\sup_{j\ge 1} I_j(\nu_t^j).
\]
On one side, the convexity gives
\[
\int_{\bfP(\bbR^d)}I_1(g)\nu_t(dg)
=\int_{\bfP(\bbR^d)}\frac{1}{j}I(g^{\otimes j})\nu_t(dg)
\ge \frac{1}{j}I\left(\int_{\bfP(\bbR^d)}g^{\otimes j}\nu_t(dg)\right)=I_j(\nu_t^j).
\]
 The other side is more tricky. One uses a type of affine property for the functional $\nu\mapsto \sup_{j\ge 1} I_j(\nu)$, and we refer the readers to \cite[Theorems 5.4, 5.7]{hauray2014}.

 Then, using \eqref{eq:entropyrelations}, it is clear that $\lim_{N\to\infty}I_j(\rho^{(j),N})
 \le \lim_{N\to\infty}I_N(f_t^N)$. The lower semicontinuity then implies $I_j(\nu_t^j)
 \le \lim_{N\to\infty}I_j(\rho^{(j),N})$.

 For the entropy, it is shown in \cite[Lemma 4.2]{fournier2014} that $H_j$ is convex, lower semi-continuous and a certain affine property. $\bbE(H_1(\mu_t))=\int_{\bfP(\bbR^d)}H_1(g)\nu_t(dg)=\sup_{j\ge 1} H_j(\nu_t^j)$ holds.

 Since the entropy could be negative, we should use the fact that the second moment of $\rho^{(j),N}$ is uniformly bounded and \eqref{lbentropy}. We apply \eqref{lbentropy} in \eqref{eq:aux1} and have $H_j(\rho^{(j),N})
 \le (1+\frac{N-mj}{mj})H_N(f_t^N)+\frac{N-mj}{mj}\left(M_2(\rho^{(N-mj),N})+C_{j,2}\right)$. It then follows that $\lim_{N\to\infty}H_j(\rho^{(j),N})
 \le \lim_{N\to\infty}H_N(f_t^N)$ still holds. The lower semicontinuity then gives the desired result.

 (iv)  By \eqref{eq:uniformfisher}, we obtain \eqref{fisher:est2}.
Now, since $\int_0^TI_1(\mu_s)\,ds<\infty$ a.s.,  the definition of Fisher information (Equation \eqref{eq:entropy}) implies that for such $\omega$, $(\mu_s)_{s\in [0, T]}$ has density for a.e. $s\in [0, T]$. The claim for $t=0$ is a simple consequence of law of large numbers.

\end{proof}

\section{The limit measure almost surely is a weak solution}\label{sec:limitmeasure}

Now we define the weak solution of \eqref{eq:pnp} in the following sense:
\begin{definition}\label{weaksol:def}
We say $\rho\in L^{\infty}(0, T; L^1(\bbR^d))$ is a weak solution to \eqref{eq:pnp} if
\begin{itemize}
\item $\rho\,dx \in C([0, T]; C_b(\bbR^d)')$ and $\rho\nabla h\in L^1(0, T; L^1(\bbR^d))$, where $h=g*\rho$.

\item For all $t\in [0, T]$:
\begin{equation}\label{weaksol}
    \langle\rho_t,\phi\rangle-\langle \rho_0,\phi\rangle-\int_0^t\int_{\bbR^{d}}\nabla\phi(x)\cdot \nabla h(x)\rho_s(x)\,dx ds-\int_0^t\langle\rho_s,\Delta\phi\rangle ds=0
\end{equation}
for any $\phi\in C_c^2(\bbR^d)$.
\end{itemize}
\end{definition}

We first of all prove the following important result for $d=3$.
\begin{proposition}\label{asweak1}
Let $d=3$ and $\{(X_t^{i,N})\}_{i=1}^N$ be the unique solution to \eqref{model} with the i.i.d initial data $\{X_0^{i,N}\}_{i=1}^N$. Suppose the common density $\rho_0$ satisfies $H(\rho_0)<\infty$, $m_2(\rho_0)<\infty$ and $\mathcal{E}(\rho_0)<\infty$. Assume the random measure $\mu$ on $C([0,T];\bbR^3)$ is a limit point of $\mu^N$ under the topology induced by convergence in law. Then, $\mu$ has a density $(\rho_s)_{s\in [0, T]}$ a.s. as we have seen, and for fixed $\phi\in C_b^2(\bbR^3)$ and $t\in [0,T]$, $\rho$ satisfies the following integral equation almost surely.
\begin{multline}\label{weaksol1}
    \langle\rho_t,\phi\rangle-\langle \rho_0,\phi\rangle-\frac{1}{2}\int_0^t\int_{\bbR^{3}\times \bbR^3}(\nabla\phi(x)-\nabla\phi(y))\\
    \cdot F(x-y)\rho_s(x)\rho_s(y)\,dxdy\,ds-\int_0^t\langle\rho_s,\Delta\phi\rangle ds=0.
\end{multline}
\end{proposition}
\begin{proof}
We divide our proof into the following steps. 

\textbf{Step 1} The integral \eqref{weaksol1} involves the singularity, therefore we need to show that it is well-defined. Since $\phi\in C_b^2(\bbR^3)$, we only need to show that the third term is integrable for a.s. $\oino$. By the Tonelli's theorem, it suffices to show that
\begin{equation}\label{step1.1}
    \bbE\left[\frac{1}{2}\int_0^t\int_{\bbR^{2d}}|(\nabla\phi(x)-\nabla\phi(y))\cdot F(x-y)|\mu_s(dx)\mu_s(dy)ds\right]<\infty
\end{equation}
Since $\phi\in C_b^2(\bbR^3)$, by Lemma \ref{fisherIII}, we take $d=3$ and $\frac{1}{3}<\beta\le\frac{2}{3}$, there exists a constant $C$ depending only on $\phi,T$ and $\beta$ such that
\begin{equation}\label{step1.2}
    \begin{split}
        &\quad \bbE\left[\frac{1}{2}\int_0^t\int_{\bbR^{3}\times \bbR^3}|(\nabla\phi(x)-\nabla\phi(y))\cdot F(x-y)|\mu_s(dx)\mu_s(dy)ds \right]\\
        &\le2\|\nabla^2\phi\|_\infty\bbE \left[ \Big|\int_0^t\int_{\bbR^{3}\times\bbR^3}\frac{\rho_s(x)\rho_s(y)}{|x-y|}dxdyds \Big| \right]\\
        &\le C\bbE \left[\int_0^T (I_2(\rho_s^{\otimes 2})^{\frac{3\beta }{2}}+1)ds \right]
        =C\bbE \left[\int_0^T I_1(\rho_s)^{\frac{3\beta }{2}}ds \right]+CT.
    \end{split}
\end{equation}
Since $\beta<\frac{2}{3}$, using the H\"{o}lder's inequality, there exists a constant $C=C(\phi,T,\beta)$ such that
\begin{equation}\label{step1.3}
    \bbE\left[\int_0^T I_1(\rho_s)^{\frac{3\beta }{2}}ds\right]\le C \left(\int_0^T I(\rho_s)ds \right)^{\frac{3\beta}{2}}.
\end{equation}
Combining \eqref{fisher:est2}, \eqref{step1.2} and \eqref{step1.3} together we obtain \eqref{step1.1}, which means that the integral \eqref{weaksol1} is well-defined.

Now from Lemma \ref{substitution} and \eqref{substitutionII}, we can rewrite the integral \eqref{weaksol1} as
\begin{equation}\label{step1.4}
\begin{split}
&\langle\mu_t,\phi\rangle-\langle \rho_0,\phi\rangle-\frac{1}{2}\int_0^t\int_{\bbR^{3}\times\bbR^3}(\nabla\phi(x)-\nabla\phi(y))\cdot F(x-y)\mu_s(dx)\mu_s(dy)ds-\int_0^t\langle\mu_s,\Delta\phi\rangle ds\\
    =&-\frac{1}{2}\int_0^t\iint_{(C([0, T]; \bbR^3))^2}(\nabla\phi(X_s)-\nabla\phi(Y_s))\cdot F(X_s-Y_s)\mu(dX)\mu(dY)ds\\
    &+\int_{C([0, T]; \bbR^3)}(\phi(X_t)-\phi(X_0))\mu(dX)-\int_{C([0, T]; \bbR^3)}\int_0^t\Delta\phi(X_s)ds\mu(dX).
\end{split}
\end{equation}
For $X, Y\in C([0,T];\bbR^3)$, we define the functional
\begin{equation}\label{step1.5}
    \psi(X,Y)=\phi(X_t)-\phi(X_0)-\frac{1}{2}\int_0^t(\nabla\phi(X_s)-\nabla\phi(Y_s))\cdot F(X_s-Y_s)ds-\int_0^t\Delta\phi(X_s)ds
\end{equation}
and similarly $\psi_{\e}(X, Y)$ is the functional with $F$ being replaced by $F_{\e}$.
We also define functional $\calk_\psi$ and $\calk_{\psi_\epsilon}$ on $\mathbf{P}(C([0,T];\bbR^d))$ by
\begin{equation}\label{step1.7}
\begin{split}
    \calk_{\psi}(\nu)&=\int_{(C([0, T];\bbR^3))^2}\psi(X,Y)\nu(dX)\nu(dY),\\
    \calk_{\psi_{\epsilon}}(\nu)&=\int_{(C([0, T];\bbR^3))^2}\psi_{\epsilon}(X,Y)\nu(dX)\nu(dY).
\end{split}
\end{equation}
If $\nu$ is a random measure, i.e. a (measurable) mapping from $(\Omega, \mathcal{F}, \bbP)$ to $\mathbf{P}(C([0,T];\bbR^3))$, then $\calk_{\psi}(\nu)$ is a random variable on $(\Omega, \mathcal{F}, \bbP)$. Since we can change the order of integration in light of \eqref{step1.1}, from the definition \eqref{step1.7} we see that for a.s. $\oino$, the left side of \eqref{weaksol1} is actually equal to $\calk_\psi(\mu(\omega))$. Therefore it suffices to show that $\bbE[|\calk_\psi(\mu)|]=0$.

For $\e>0$, one certainly has
\begin{equation}\label{step1.9}
    \bbE[|\calk_\psi(\mu)|]\le\bbE[|\calk_\psi(\mu)-\calk_{\psi_\epsilon}(\mu)|]+\bbE[|\calk_{\psi_\epsilon}(\mu)|].
\end{equation}
In the following steps we show that each term of \eqref{step1.9} goes to 0 as $\e\to0$.

\textbf{Step 2} Now we investigate the first term of \eqref{step1.9}. For fixed $\omega\in\Omega$, $\mu$ is a probability measure on $C([0,T],\bbR^3)$ thus we can apply Lemma \ref{substitution} and obtain
\begin{equation}\label{step2.1}
\begin{split}
    &\bbE[|\calk_\psi(\mu)-\calk_{\psi_\epsilon}(\mu)|]\\
    &=\frac{1}{2}\bbE\left[\Big|\int_0^t\int_{C\times C}[\nabla\phi(X_s)-\nabla\phi(Y_s)]\cdot[F_\epsilon(X_s-Y_s)-F(X_s-Y_s)]\mu(dX)\mu(dY)ds\Big|\right]\\
    &=\frac{1}{2}\bbE\left[ \Big|\int_0^t\iint_{\bbR^3\times\bbR^3}[\nabla \phi(x)-\nabla\phi(y)]\cdot[F_\epsilon(x-y)-F(x-y)]\mu_s(dx)\mu_s(dy)ds \Big| \right]\\
    &=\frac{1}{2}\bbE\left[ \Big|\int_0^t\iint_{|x-y|<\epsilon}[\nabla\phi(x)-\nabla\phi(y)]\cdot[F_\epsilon(x-y)-F(x-y)]\rho_s(x)\rho_s(y)dxdyds \Big| \right]\\
    &\le C_d\bbE\|\nabla^2\phi\|_\infty\int_0^t\iint_{|x-y|<\epsilon}\frac{\rho_s(x)\rho_s(y)}{|x-y|}dxdyds.
    \end{split}
\end{equation}
In the equation we used the fact that $|F_\epsilon(x)|\le|F(x)|$ and $F_\epsilon(x)=F(x)$ when $|x|\ge\epsilon$. Now, we apply \eqref{fisher:exp2} in Lemma \ref{fisherIII} by taking $\gamma=1$ and obtain
\begin{equation}\label{step2.2}
\begin{split}
\bbE\int_0^t\iint_{|x-y|<\epsilon}\frac{\rho_s(x)\rho_s(y)}{|x-y|}dxdyds&\le C_\beta\epsilon^{3\beta-1}\bbE\int_0^t I(\rho_s^{\otimes 2})^{\frac{3\beta}{2}}ds  \\
&\le C(\beta,T)\epsilon^{3\beta-1}\left(\bbE\int_0^t I(\rho_s)ds \right)^{\frac{3\beta}{2}},
\end{split}
\end{equation}
where $\frac{1}{3}<\beta<\frac{2}{3}$. Therefore by \eqref{fisher:est2}, there exists $C=C(\phi,T,\beta, \rho_0)$ such that
\begin{equation}\label{step2.3}
    \bbE[|\calk_\psi(\mu)-\calk_{\psi_\epsilon}(\mu)|]\le C \epsilon^{3\beta-1}.
\end{equation}

\textbf{Step 3} For the second term of \eqref{step1.9}, since $\psi_\epsilon$ is bounded and continuous on $C([0,T];\bbR^3)\times C([0,T];\bbR^3)$ and $\mu^N\to\mu$ in law, applying Lemma \ref{convinlaw} with $\mathcal{X}=C([0, T];\bbR^3)$, the random variable $\calk_{\psi_\epsilon}(\mu^N)$ converges to $\calk_{\psi_\epsilon}(\mu)$ in law for fixed $\epsilon$.
Since $\calk_{\psi_\epsilon}(\mu^N)$ and $\calk_{\psi_\epsilon}(\mu)$ are bounded by $\|\psi_{\epsilon}\|_{L^{\infty}}$, we can take $\phi(x)=|x|\wedge \|\psi_{\epsilon}\|_{L^{\infty}}$ as the test function and conclude
\begin{equation}\label{step3.1}
    \lim_{N\to\infty}\bbE[|\calk_{\psi_\epsilon}(\mu^N)|]=\bbE[|\calk_{\psi_\epsilon}(\mu)|].
\end{equation}

Now we investigate $\bbE[|\calk_{\psi_\epsilon}(\mu^N)|]$. By definition, it holds that
\begin{multline}\label{step3.2}
\calk_{\psi_\epsilon}(\mu^N)=\frac{1}{N^2}\sum_{i,j=1}^N\psi_\epsilon(X_t^{i,N},X_t^{j,N})
  =\frac{1}{N^2}\sum_{i,j=1}^N\Big[\phi(X_t^{i,N})-\phi(X_0^{i,N}) \\
  -\frac{1}{2}\int_0^t(\nabla\phi(X_s^{i,N})-\nabla\phi(X_s^{j,N}))\cdot F_{\epsilon}(X_s^{i,N}-X_s^{j,N})ds-\int_0^t\Delta\phi(X_s^{i,N})ds\Big].
\end{multline}
Now we apply the It\^{o}'s formula to $\phi\in C_b^2(\bbR^3)$ and obtain
\begin{multline}\label{step3.3}
\sum_{i=1}^N(\phi(X_t^{i,N})-\phi(X_0^{i,N}))=\frac{1}{N}\sum_{i=1}^N\sum_{k=1,k\neq i}^N\int_0^t\nabla\phi(X_s^{i,N})\cdot F(X_s^{i,N}-X_s^{k,N})ds \\
+\sum_{i=1}^N\int_0^t\Delta\phi(X_s^{i,N})ds
    +\sqrt{2}\sum_{i=1}^N\int_0^t\nabla\phi(X_s^{i,N})\cdot dB_s^i .
\end{multline}
Note that by symmetry
\begin{multline}\label{step3.4}
\sum_{j=1}^N\frac{1}{N}\sum_{i=1}^N\sum_{k=1,k\neq i}^N\int_0^t\nabla\phi(X_s^{i,N})\cdot F(X_s^{i,N}-X_s^{k,N})ds \\
    =\frac{1}{2}\sum_{i=1}^N\sum_{j=1,j\neq i}^N\int_0^t(\nabla\phi(X_s^{i,N})-\nabla\phi(X_s^{j,N}))\cdot F(X_s^{i,N}-X_s^{j,N})ds .
\end{multline}
Therefore, one has
\begin{multline}\label{step3.5}
    \calk_{\psi_\epsilon}(\mu^N)=\frac{1}{2N^2}\sum_{i=1}^N\sum_{j=1,j\neq i}^N\int_0^t\Big[(\nabla\phi(X_s^{i,N})-\nabla\phi(X_s^{j,N}))\\
    \cdot (F(X_s^{i,N}-X_s^{j,N})-F_{\epsilon}(X_s^{i,N}-X_s^{j,N}))\Big]ds
    +\frac{\sqrt{2}}{N}\sum_{i=1}^N\int_0^t\nabla\phi(X_s^{i,N})\cdot dB_s^i
\end{multline}
and thus
\begin{multline}\label{step3.6}
\bbE[|\calk_{\psi_\epsilon}(\mu^N)|]\le\frac{N-1}{2N}\int_0^t\iint_{\bbR^{3}\times\bbR^3}\rho_s^{(2),N}(x,y) \Big|(\nabla\phi(x)-\nabla\phi(y))\\
\cdot(F(x-y)-F_{\epsilon}(x-y)) \Big|dxdyds
        +\bbE\left[\Big|\frac{\sqrt{2}}{N}\sum_{i=1}^N\int_0^t\nabla\phi(X_s^{i,N})\cdot dB_s^i\Big|\right].
\end{multline}
Again note that $F_\epsilon(x)=F(x)$ when $|x|\ge\epsilon$, one has
\begin{equation}\label{step3.7}
    \begin{split}
        &\int_0^t\iint_{\bbR^{3}\times\bbR^3}\rho_s^{(2),N}(x,y)|(\nabla\phi(x)-\nabla\phi(y))\cdot(F(x-y)-F_{\epsilon}(x-y))|dxdyds\\
        &\le 2C_d\|\nabla^2\phi\|_{\infty}\int_0^t\iint_{|x-y|<\epsilon}\frac{\rho_s^{(2),N}(x,y)}{|x-y|}dxdyds.
    \end{split}
\end{equation}
Applying Lemma \ref{fisherIII} once more with $\frac{1}{3}<\beta<\frac{2}{3}$ and $\gamma=1$, one obtains
\begin{equation}\label{step3.8}
    \begin{split}
        &\int_0^t\iint_{|x-y|<\epsilon}\frac{\rho_s^{(2),N}(x,y)}{|x-y|}dxdyds\le C_\beta\epsilon^{3\beta-1}\int_0^t I(\rho_s^{(2),N})^{\frac{3\beta}{2}}ds\\
        &\le C(\beta, T)\epsilon^{3\beta-1} \left(\int_0^t I(\rho_s^{(2),N}(x,y))ds \right)^{\frac{3\beta}{2}}\le C \epsilon^{3\beta-1}.
    \end{split}
\end{equation}
Here the constant $C=C(\phi,T,\beta, \rho_0)$ comes from \eqref{eq:uniformmarginal} for $j=2$.

For the second term of \eqref{step3.6}, from the independence of the Brownian motions $\{B_t^i\}_{i=1}^N$, one can easily calculate its second moment
\begin{equation}\label{step3.9}
    \begin{split}
        \bbE \left[\Big|\frac{\sqrt{2}}{N}\sum_{i=1}^N\int_0^t\nabla\phi(X_s^{i,N})\cdot dB_s^i \Big|^2 \right]&=\frac{2}{N^2}\sum_{i=1}^N\int_0^t\bbE[|\nabla\phi(X_s^{i,N})|^2]ds
        \le\frac{2}{N}T\|\nabla \phi\|_{\infty}^2,
    \end{split}
\end{equation}
which implies
\begin{equation}\label{step3.10}
    \bbE \left[\Big|\frac{\sqrt{2}}{N}\sum_{i=1}^N\int_0^t\nabla\phi(X_s^{i,N})\cdot dB_s^i \Big| \right]
    \le\frac{C}{\sqrt{N}}.
\end{equation}

Plugging \eqref{step3.8} and \eqref{step3.10} into \eqref{step3.6}, one finds that for any $\epsilon,N>0$,
\begin{equation}\label{step3.11}
    \bbE[|\calk_{\psi_\epsilon}(\mu^N)|]\le C\left(\frac{1}{\sqrt{N}}+\epsilon^{3\beta-1}\right).
\end{equation}
\textbf{Step 4} Finally, we combine the estimates above together.

Plugging \eqref{step3.11} into \eqref{step3.1}, one finds that
\begin{equation}\label{step4.1}
    \bbE[|\calk_{\psi_\epsilon}(\mu)|]=\lim_{N\to\infty}\bbE[|\calk_{\psi_\epsilon}(\mu^N)|]\le C\epsilon^{3\beta-1}.
\end{equation}
By \eqref{step2.3},
\begin{equation}\label{step4.2}
    \bbE[|\calk_\psi(\mu)-\calk_{\psi_\epsilon}(\mu)|]\le C \epsilon^{3\beta-1}.
\end{equation}
Finally by \eqref{step4.1} and \eqref{step4.2}
\begin{equation}\label{step4.3}
    \bbE[|\calk_\psi(\mu)|]\le\liminf_{\epsilon\to0}(\bbE[|\calk_\psi(\mu)-\calk_{\psi_\epsilon}(\mu)|]+\bbE[|\calk_{\psi_\epsilon}(\mu)|])=0.
\end{equation}
This is the desired conclusion.
\end{proof}

Note that in \eqref{weaksol1}, we have symmetrized equation \eqref{weaksol}.
This symmetrization technique reduces the singularity from $|x|^{-3}$ to $|x|^{-2}$ so that Lemma \ref{fisherIII} can be applied to control the singularity. This is one of the important observations in this work. The bottleneck for general $d$ is that the singularity allowed in Lemma \ref{fisherIII} is only $(0, 2)$ for all $d$. In fact, for $d\ge4$ cases, Proposition \ref{pro:uniformboundenergy} below (the proof does not rely on $d$) actually implies that  \eqref{step1.1} still holds. Therefore, Step 1 and Step 2 of the proof is still valid with \eqref{step2.3} replaced by 
$\lim_{\e\to0}\bbE[|\calk_\psi(\mu)-\calk_{\psi_\epsilon}(\mu)|] = 0$.
However, the difficulty arises from the $N$ particle system \eqref{step3.8}, where the Fisher Information no longer provides the uniform estimate and we know nothing about 
\[
\lim_{\e\to 0}\lim_{N\to\infty}\int_0^t\iint_{|x-y|<\epsilon}\frac{\rho_s^{(2),N}(x,y)}{|x-y|^{d-2}}dxdyds.
\] Recalling the proof of Lemma \ref{fisherIII}, if we can find better uniform $L^p$ estimates for the density of $X_t^{i,N}-X_t^{j,N}$, then one might be able to pass the limit for $d\ge4$ cases.
 Hence, we think for general $d\ge 4$, the entropy way does not work unless new estimates are found.

We now give some $L^p$ estimates for the density $\rho$ of the limit measure $\mu$. For the convenience, we will then reserve $h$ as
\begin{gather}
h=(-\Delta)^{-1}\rho=g*\rho.
\end{gather}
\begin{lemma}\label{limitmeasure:lp}
Let $d=3$. Let $\rho$ be the density of the (random) limit measure $\mu$. Then for a.s. $\oino$, we have the following claims:
\[
\begin{split}
\rho &\in L^{\frac{2q}{4q-3}}(0, T; \dot{W}^{1,q})\cap L^{\frac{2p}{3(p-1)}}(0, T; L^p),~~ p\in [1,3],q\in [1, 3/2],\\
\nabla h &\in L^{2p_1/(2p_1-3)}(0, T; L^{p_1}),~p_1\in (3/2, \infty).
\end{split}
\]
Consequently, $\rho\nabla h\in L^1(0, T; L^1(\bbR^3))$ a.s..
\end{lemma}
The claims for $\rho$ follow from equation \eqref{fisher:est2} and Lemma \ref{fisherII}.
The claims of $\nabla h$ are due to the Hardy-Littlewood-Sobolev inequality since $\nabla h=\nabla g*\rho$. We skip the details.

By Proposition \ref{asweak1} and Lemma \ref{limitmeasure:lp}, now we are able to prove that the density of the limit measure $\mu$ is a.s. a weak solution for the equation \eqref{eq:pnp}.

\begin{theorem}\label{asweak2}
Suppose $d=3$ and $\rho_0,\mu_N,\mu, \rho$ satisfy the assumptions in Proposition \ref{asweak1}, i.e., the common density $\rho_0$ satisfies $H(\rho_0)<\infty$, $m_2(\rho_0)<\infty$ and $\mathcal{E}(\rho_0)<\infty$, while the random measure $\mu$ on $C([0,T];\bbR^3)$ is a limit point of $\mu^N$ under the topology induced by convergence in law with a.s. density $\rho$. Then for a.s. $\oino$, $\rho$ is a weak solution to the nonlinear Fokker-Planck equation \eqref{eq:pnp}.
\end{theorem}

\begin{proof}
First we fix $\phi$ and show that \eqref{weaksol1} holds for all $t\in [0,T]$ and a.s. $\oino$. In fact, by Proposition \ref{asweak1} and \eqref{step1.2}, the following set has probability $1$:
\begin{multline}
A=\Big\{\oino|\int_0^t\int_{\bbR^{3}\times \bbR^3}|(\nabla\phi(x)-\nabla\phi(y))\cdot F(x-y)|\mu_s(dx)\mu_s(dy)ds<\infty,\\
 \text{\eqref{weaksol1} holds for } t\in[0,T]\cap \mathbb{Q}  \Big\}.
\end{multline}
For any probability measure $\mu\in\bbP(C([0,T];\bbR^3)))$, $\psi\in C_b(\bbR^3)$ and  $t_n\to t$ we may apply the dominant convergence theorem
	\begin{equation}
	\int_{C([0,T];\bbR^3)}\psi(X_{t_n})\mu(dX)\to \int_{C([0,T];\bbR^3)}\psi(X_{t})\mu(dX),
	\end{equation}
	which gives
	\begin{equation}\label{weakl1cont}
	\lim_{n\to\infty}\int_{\bbR^{3}}\psi(x)\mu_{t_n}(dx)=\int_{\bbR^{3}}\psi(x)\mu_{t}(dx)
	\end{equation}
	by Lemma \ref{substitution}. From \eqref{weakl1cont} we see that both $\inner{\Delta\phi}{\mu_t}$ and $\inner{\phi}{\mu_t}$ are continuous functions on $[0,T]$. The continuity them implies that for $\omega\in A$, \eqref{weaksol1} holds for all $t\in [0, T]$.
	
Now we show that for a.s. $\oino$, $\rho(\omega)$ satisfies \eqref{weaksol1} both for all $t\in[0,T]$ and all $\phi\in C_c^2(\bbR^3)$. In fact, since $C_c^2(\bbR^3)$ is separable (note that $C_b^2(\bbR^3)$ is not separable), there is a countable dense set $\{\phi_n\}$. Then, for a.s. $\oino$, $\mu(\omega)$ satisfies \eqref{weaksol1} for all $t\in[0,T]$ and $\phi=\phi_n$. Now in light of \eqref{step1.2}, for a.s. $\oino$, the left side of \eqref{weaksol1} can be viewed as a bounded linear functional on $C_c^2(\bbR^3)$. The conclusion then follows from the density of $\{\phi_n\}$.

Lastly, $\rho \nabla h\in L^1(0, T; L^1(\bbR^3))$ from Lemma \ref{limitmeasure:lp}, and we can then change the symmetric integral equation \eqref{weaksol1} into the usual one
\eqref{weaksol}.
\end{proof}

The weak solution defined above has the minimal regularity requirement. In fact, the system we consider could give more information and we can improve the regularity.  We first of all have the following important claim about the energy.
\begin{proposition}\label{pro:uniformboundenergy}
Consider a general dimension $d\ge 3$. Suppose $\mu$ is any limit point of the $\mu^N$ which a.s. has density as we have seen. Then for a.s. $\oino$, the energy
\begin{gather}
\mathscr{E}(t , \omega):=\iint_{\bbR^d\times\bbR^d} g(x-y)\mu_t(dx)\mu_t(dy)
\end{gather}
is bounded by the initial energy:
\[
\sup_{t\in [0, T]}\mathscr{E}(t, \omega)\le \mathcal{E}(\rho_0)=\iint_{\bbR^d\times\bbR^d} g(x-y)\rho_0(x)\rho_0(y)\,dxdy.
\]
\end{proposition}

\begin{proof}
From \eqref{noncollision6}, it holds that
\begin{equation}\label{energy:finite1}
	\sup_{0\le t\le T}\frac{1}{N^2}\sum_{i,j=1,i\neq j}^Ng^\e(X_t^{i,\e}-X_t^{j,\e})\le \frac{1}{N^2}\sum_{i,j=1,i\neq j}^Ng^{\e}(X_0^i-X_0^j)+\frac{1}{N}\sup_{0\le t\le T}(-M_t^\e).
\end{equation}

Since for fixed $N$, $g^\e(X_t^{i,\e}-X_t^{j,\e})=g(X_t^{i}-X_t^{j})$ for all $t\in [0, T]$ outside a set $A_{\e}$ whose probability goes to zero as $\e\to 0$ by the non-collision result, we then have
almost surely that
\[
\sup_{0\le t\le T}\frac{1}{N^2}\sum_{i,j=1,i\neq j}^Ng(X_t^{i}-X_t^{j}) =\lim_{\epsilon\to0}\sup_{0\le t\le T}\frac{1}{N^2}\sum_{i,j=1,i\neq j}^Ng^\e(X_t^{i,\e}-X_t^{j,\e}).
\]
Fatou's lemma gives us that
\begin{gather}
\begin{split}
& \mathbb{E}\left[\sup_{0\le t\le T}\frac{1}{N^2}
\sum_{i,j: i\neq j}g(X_t^i-X_t^j)-\mathcal{E}(\rho_0) \right]^+ \\
&\le \liminf_{\e\to 0}\mathbb{E}\left( \frac{1}{N^2}\sum_{i,j:i\neq j}g^{\e}(X_0^i-X_0^j)-\mathcal{E}(\rho_0) \right)^+
+\liminf_{\e\to 0}\frac{1}{N}\mathbb{E}\Big[\sup_{0\le t\le T}(-M_t^\e)\Big]^+.
\end{split}
\end{gather}

Doob's $L^p$ inequality for martingale  (p.203, Theorem 7.31 in \cite{klebaner2012}) and \eqref{noncollision7} imply that
\begin{equation}\label{energy:finite2}
\frac{1}{N}\left\|\sup_{0\le t\le T}(-M_t^\e) \right\|_{L^2(\mathbb{P})}\le
\frac{1}{N}\left(\mathbb{E}\sup_{0\le t\le T}(-M_t^{\e})^2 \right)^{1/2}\le \frac{2}{N}\bbE[(M_T^\e)^2]^\frac{1}{2}\le \frac{C_{d,\rho_0}}{\sqrt{N}}.
\end{equation}
Hence, the last term goes to zero as $N\to \infty$.
Moreover, at $t=0$, the joint distribution of $(X_0^i, X_0^j, X_0^m, X_0^n)$ is simply $\rho_0^{\otimes 4}$ if they are all distinct. In the square of $\sum_{i,j:i\neq j}g^{\e}(X_0^i-X_0^j)$, the number of terms where some $X_0^i$'s are repeated is $O(N^3)$. Hence, most terms are those where the four $X_0^i$'s are distinct. Using this fact and direct computation, we find
\begin{multline}
\lim_{N\to\infty}\liminf_{\e\to 0}\mathbb{E}\Big(\frac{1}{N^2}\sum_{i,j:i\neq j}g^{\e}(X_0^i-X_0^j)-\mathcal{E}(\rho_0)\Big)^+ \\
\le \lim_{N\to\infty}\liminf_{\e\to 0}\left(\bbE\Big|\frac{1}{N^2}\sum_{i,j:i\neq j}g^{\e}(X_0^i-X_0^j)-\mathcal{E}(\rho_0) \Big|^2 \right)^{\frac{1}{2}}=0.
\end{multline}

It follows that
\[
\lim_{N\to\infty}\mathbb{E} \Big[\sup_{0\le t\le T}\frac{1}{N^2}
\sum_{i,j: i\neq j}g(X_t^i-X_t^j)-\mathcal{E}(\rho_0) \Big]^+=0.
\]

Now, for $\nu\in \bfP(C[0,T],\bbR^d)$, we define
\begin{gather}
Q_t(\nu):=\iint_{\mathcal{D}^c}g(x-y)\nu_t(dx)\nu_t(dy),
\end{gather}
where $\nu_t$ is defined in \eqref{eq:mut}.

We also define
\begin{gather}
Q(\nu):=\sup_{0\le t\le T}Q_t(\nu)= \sup_{0\le t\le T}\iint_{\mathcal{D}^c}g(x-y)\nu_t(dx)\nu_t(dy).
\end{gather}
We claim that if we consider the topology induced by weak convergence on $\bfP(C[0,T],\bbR^d)$, then $Q(\nu)$ is a lower semi-continuous functional on $\bfP(C[0,T],\bbR^d)$.  In fact,  we can define
\[
Q_t^{(m)}(\nu):=\iint_{(C([0,T];\bbR^d))^2}
g^{(m)}(X(t)-Y(t)) \nu(dX) \nu(dY),
\]
where $g^{(m)}(x)=g(|x|)$ if $|x|\ge 1/m$ and $g^{(m)}(x)=m|x| g(1/m)$ if $|x|<1/m$. Since $g^{(m)}$ is a continuous bounded function, by \cite[Theorem 2.8]{billingsley2013} (p. 23), $Q^{(m)}$ is a continuous functional.
Moreover, by monotone convergence theorem, $\forall \nu\in \bfP(C[0,T],\bbR^d)$,
$Q_t(\nu)=\sup_{m}Q^{(m)}(\nu)$.
Hence, $Q_t$ is lower semicontinuous, and thus $Q=\sup_t Q_t$ is lower semicontinuous.

From previous proof, a subsequence of $\{\mu_N\}$ converges in law to some random measure $\mu$. Since $\bfP(C[0,T],\bbR^d)$ is now a Polish space, from \cite[p.415, Theorem 11.7.2]{dudley2002} there exists some probability space $(\tilde{\Omega},\tilde{\mathcal{F}}, \tilde{\bbP})$ and random measures $\tilde{\mu}_N,\tilde{\mu}: \ (\tilde{\Omega},\tilde{\mathcal{F}}, \tilde{\bbP})\to\bfP(C[0,T],\bbR^d)$ such that $\tilde{\mu}_N\to\tilde{\mu}$ a.s., and $\tilde{\mu}_N,\tilde{\mu}$ has the same law as $\mu_N,\mu$. By the Fatou Lemma and the lower semicontinuity, we have
\begin{multline}
\bbE[(Q(\mu)-\mathcal{E}(\rho_0))^+]=\tilde{\bbE}[(Q(\tilde{\mu})-\mathcal{E}(\rho_0))^+]\\
\le\liminf_{N\to\infty}\tilde{\bbE}[
(Q(\tilde{\mu}_N)-\mathcal{E}(\rho_0))^+]=\liminf_{N\to\infty}\bbE[(Q(\mu_N)-\mathcal{E}(\rho_0))^+]=0.
\end{multline}

Moreover, since $\mu$ has density almost surely, then almost surely it holds that
\begin{equation}
\sup_{0\le t\le T}\iint_{\bbR^d\times\bbR^d} g(x-y)\mu_t(dx)\mu_t(dy) \le \mathcal{E}(\rho_0).
\end{equation}
\end{proof}

With the above estimate, $\rho\in L^{\infty}(0, T; H^{-1})$ and $\nabla h\in L^{\infty}(0, T; L^2)$. Then, we have the following improved weak solution, and we provide the proof in Appendix \ref{app:missingpf}.
\begin{proposition}\label{pro:improvedweak}
Let $d=3$. Suppose $\mu(\cdot)$ is a time-dependent probability measure, which has a density $\rho$.
Assume that $\rho$ is a weak solution to \eqref{eq:pnp} in the sense of Definition \ref{weaksol:def}.
If moreover, $\int_0^TI(\mu_t)\,dt<\infty$ and $\sup_{t\in [0, T]}\iint_{\mathbb{R}^3\times\mathbb{R}^3} g(x-y)\mu_t(dx)\mu_t(dy)\le \mathcal{E}(\rho_0)$,
then
\begin{enumerate}
\item $\rho\in L^{3r/(5r-6)}(0, T; L^r)$ for $r\in [3/2, 3]$; $\nabla h\in L^{q/(q-2)}(0, T; L^q)$ for $q\ge 2$. Consequently, $\rho\nabla h$ is in $ L^{3p/(8p-6)}(L^p)$ for $1\le p\le\frac{6}{5}$ (recall $h=g*\rho$).

\item In $L^{6/5}((0, T), W^{-1, 12/11})$, it holds that
\begin{gather}
\partial_t\rho=\nabla\cdot(\rho\nabla h)+\Delta\rho.
\end{gather}
Moreover, $\rho$ is a mild solution in $L^{4/3}(0, T; L^{3/2}(\bbR^3))$ so that
\begin{equation}\label{mildsolform}
\rho(t)=e^{t\Delta}\rho_0+\int_0^t e^{(t-s)\Delta}\nabla\cdot(\rho\nabla h)\,ds.
\end{equation}
\end{enumerate}
\end{proposition}
Note that the mild solution form here does not necessarily give the continuity of $\rho(t)$ at $t=0$ because we do not know whether the second term goes to $0$ as $t\to 0^+$.

\section{A comment about propagation of chaos in 3D}\label{sec:3dchaos}

We have established the fact that the limit measure is almost surely a weak solution to the nonlinear Fokker-Planck equation \eqref{eq:pnp}.  An important question in the mean field limit research is whether we have propagation of chaos. In other words, we expect the $j$-marginal tends to the tensor product of the limit law $\rho$.  First we recall the following standard equivalent notions of propagation of chaos which can be found in the lecture of Sznitman (Proposition 2.2 in \cite{sznitman1991}).
\begin{definition}
Let $\mathcal{X}$ be a Polish space and $f$ be a probability measure on $\mathcal{X}$. A sequence of symmetric probability measures $f^N$ on $\mathcal{X}^N$ are said to be $f-$chaotic, if one of the three following equivalent conditions is satisfied:

(i) The sequence of second marginals $f^{(2),N}\rightharpoonup f\otimes f$ as $N\to\infty$;

(ii) For all $j\ge1$, the sequence of $j$-th marginals $f^{(j),N}\rightharpoonup f^{\otimes j}$ as $N\to\infty$;

(iii) Let $(X^{1,N}, \ldots, X^{N,N})\in\mathcal{X}^N$ be drawn randomly according to $f^N$. The empirical (random) measure $\mu^N=\frac{1}{N}\sum_i \delta_{X^{i,N}}$ converges in law to the constant probability measure $f$ as $N\to\infty$.

\end{definition}

Note that since $\mathcal{X}$ is a Polish space, there exists a metric $d_0$ on $\mathbf{P}(\mathcal{X})$ such that for $\nu^N,\nu\in\mathbf{P}(\mathcal{X})$, $\nu^N\to\nu$ in law if and only if $d_0(\nu^N,\nu)\to 0$ as $N\to\infty$. Therefore as $f$ is constant, (iii) is equivalent to $\mu^N$ converging to $f$ in probability.  

The key point of propagation of chaos is therefore to establish a strong-weak uniqueness principle for the solutions so that $\rho_t(\omega)$ is a.s. deterministic.
The definition of weak solution in Definition \ref{weaksol:def} is too weak and it is very hard to prove the uniqueness. We need to put more constraints to make it unique. In fact, we have the strong-weak uniqueness principle by assuming $\rho\in L_{\loc}^2((0, T); L^2(\bbR^3))$.
\begin{proposition}\label{prop:uniquenessresult}
Let the initial density $\rho_0 \in H^m(\bbR^3)$ with $m>3/2$.
Suppose $\mu(\cdot)$ is a time-dependent probability measure, which has a density $\rho$.
Assume that $\rho$ is a weak solution to \eqref{eq:pnp} in the sense of Definition \ref{weaksol:def}.
If moreover, $\int_0^TI(\mu_t)\,dt<\infty$, $\rho\in L_{\loc}^2((0, T); L^2(\mathbb{R}^3))$ and
\[
\sup_{t\in [0, T]}\iint_{(\bbR^3)^2} g(x-y)\mu_t(dx)\mu_t(dy)\le \mathcal{E}(\rho_0)=\iint_{(\bbR^3)^2} g(x-y)\rho_0(x)\rho_0(y)\,dxdy,
\]
then $\rho_t$ is the unique strong solution of \eqref{eq:pnp}.
\end{proposition}
The proof of this proposition, though important, is tedious, and we attach it in Appendix \ref{app:missingpf}. In fact, we do not have good enough a priori $L^p$ estimates, so the usual hyper-contractivity method for Keller-Segel equations (for instance, \cite{liuwang2016,corrias2016}) will not work. What we use is an energy method appeared in \cite{serfaty2018mean}.

Recall that the energy equality \eqref{energy:est1} tells us that
\begin{gather}\label{eq:smoothtrace}
\sup_{N,\e}\int_0^T\iint_{(\bbR^3)^2} J_{\e}(x_1-x_2)\mu_{N,\e}^{(2)}(x_1, x_2)\,dx_1 dx_2 dt<\infty.
\end{gather}
Fomally, if we take $\e\to 0$, we would have
\[
\int_0^T \tilde{\rho}^N(0, t)\,dt\le C,
\]
where $\tilde{\rho}^N$ is the density for $X_1^N-X_2^N$. As $N\to\infty$, it is expected that
\[
\int_0^T\int_{\bbR^3}\mu^{(2)}(x, x)\,dx dt\le C_1.
\]
This should be
\[
\mathbb{E}\int_0^T\int_{\bbR^3}\rho^2\,dxdt\le C_1,
\] 
which is desired. 
However, rigorously justifying these limits need some uniform convergence and this seems hard. 
We will keep on working on the weak-strong uniqueness principle.

One may be tempted to send $N\to\infty$ first in
\begin{gather*}
\sup_{N,\e}\int_0^T\iint_{\bbR^3\times\bbR^3} J^{\e}(x_1-x_2)\mu_{N,\e}^{(2)}(x_1, x_2)\,dx_1 dx_2 dt<C.
\end{gather*}
The mollified system has the propagation of chaos, and the limit measure is unique, which is the strong solution $\rho_{\e}$ to the nonlinear Fokker-Planck equation with $F$ being replaced by $F_{\e}$. Though $\rho_{\e}$ has uniform $L^2(0, T; L^2)$ bound, one can show that $\rho_{\e}$ converges to the strong solution of the nonlinear Fokker-Planck equation constructed in Appendix \ref{app:strongsoln}, instead of the limit measure $\rho$. Hence, this does not work.

\begin{remark}
	Note that for $d=2$ case, the assumption $\rho\in L^2(0,T;L^2(\bbR^2))$ is a direct corollary from \eqref{fisher:lp} in Lemma \ref{fisherII} by taking $p=d=2$, and one can check that the proof of Proposition \ref{prop:uniquenessresult} is valid for $d=2$. And for Proposition \ref{asweak1}, the self-consistent martingale problem proved in Section 4 of \cite{liuyang2016} implies the conclusion. Hence combining these two results one obtains the propagation of chaos result for $d=2$ easily.
\end{remark}

\begin{remark}
In fact, in the energy estimate, one also expects the first negative term will give us
\begin{gather}\label{eqrmk:need2}
\int_0^T\int_{\bbR^3}\rho |\nabla h|^2\,dxdt<C.
\end{gather}
If this is true, many proofs can be simplified. For example, we will then have $\nabla\rho\cdot\nabla h=\nabla\sqrt{\rho}(\sqrt{\rho}\nabla h)\in L^1(0, T; L^1)$.
Then using the mild solution form \eqref{mildsolform} and the nonnegativity of $\rho^2$, one finds $\rho\in L^2(0, T; L^2)$. However, \eqref{eqrmk:need2} seems difficult to justify.
\end{remark}

With Theorem \ref{asweak2} and Proposition \ref{prop:uniquenessresult}, we conclude the following.
\begin{theorem}\label{thm:propagationchaos}
For $d=3$, let $\{(X_t^{i,N})\}_{i=1}^N$ be the unique solution to \eqref{model} with i.i.d. initial data $\{(X_0^{i,N})\}_{i=1}^N$. Suppose the common initial
density $\rho_0\in L^1(\bbR^3)\cap H^{m}(\bbR^3)$ for $m>3/2$, with $m_2(\rho_0)<\infty$, $H_1(\rho_0)<\infty$. Suppose any limit point of the empirical measure $\mu^N$ defined in \eqref{eq:empirical}
satisfies
\[
\mathbb{E}\int_0^T\int_{\bbR^3}\rho^2\,dxdt<\infty.
\]
Then, $\mu^N$ goes in probability to a deterministic measure $\bar{\mu}:=(\rho_t\,dx)_{t\in[0,T]}$ in $\mathbf{P}(C([0,T];\bbR^3))$ as $N\to\infty$, where $\rho_t$ is the unique strong solution to \eqref{eq:pnp} with initial value $\rho_0$.
\end{theorem}

\begin{proof}
We consider the metric $d_0$ on $\bfP(C([0,T],\bbR^3))$ induced by weak convergence. From Proposition \ref{tightI}, we know that $\mathcal{L}(\mu^N)$ is tight in $\mathbf{P}(\mathbf{P}(C([0,T];\bbR^3)))$. Therefore for any subsequence of $\mu^N$, there exists a further subsequence $\{\mu^{N_k}\}$ converging in law to some random measure $\mu$: $(\Omega,\mathcal{F},\bbP)\to C([0,T];\bbR^3)$. Then by Theorem \ref{asweak2}, for a.s. $\omega\in\Omega$, the limiting point $\mu$ has a density $\rho$, which is a weak solution to \eqref{eq:pnp}. By the assumption and Proposition \ref{prop:uniquenessresult}, the weak solution to \eqref{eq:pnp} is unique. Therefore if we denote $\bar{\mu}$ the (deterministic) random measure with density $\bar{\rho}$, which is the strong solution to \eqref{eq:pnp}, then for any $0\le t\le T$ and $\omega\in\Omega$,  $\mu=\bar{\mu}$ for a.s. $\omega\in\Omega$. Since the subsequence $\{\mu^{N_k}\}$ converge in law to $\mu$ and $\mu$ is a.s. equal to the deterministic probability measure $\bar{\mu}$, we see that $\mu^{N_k}$ converge in probability to $\bar{\mu}$. In other words, any subsequence of $\{\mu^N\}$ has a further subsequence $\{\mu^{N_k}\}$ converging in probability to $\bar{\mu}$. Hence, $\{\mu^N\}$ converges in probability to the deterministic probability measure $\bar{\mu}$ in $\mathbf{P}(C([0,T];\bbR^3))$.
\end{proof}

\section*{Acknowledgement}
The work of L. Li was partially sponsored by NSFC 11901389, 11971314, and Shanghai Sailing Program 19YF1421300. The work of J.-G. Liu was partially supported by KI-Net NSF RNMS11-07444 and NSF DMS-1812573. P. Yu was partially supported by the elite undergraduate training program of School of Mathematical Sciences in Peking University. 

\appendix
\section{Notes on the nonlinear Fokker-Planck equation}\label{app:strongsoln}
In this part, we investigate some properties of the nonlinear Fokker-Planck equation \eqref{eq:pnp}. We will show the local existence and uniqueness of strong solution for \eqref{eq:pnp} given the initial data is small in some space $H^s$, and then we will discuss some potential methods for the uniqueness of the weak solution.

First we state a useful lemma in \cite{biler2010}, which is some type of Banach fixed point theorem.
\begin{lemma}\label{biler}
Let (X, $\|\cdot\|_X$) be a Banach space and $H: X \times X \to X$ a bounded
bilinear form satisfying $\|H(x_1, x_2)\|_X \le \eta\|x_1\|_X \|x_2\|_X$ for all $x_1, x_2 \in X$ and a
constant $\eta > 0$. Then, if $0 <\e< \frac{1}{4\eta}$ and if $f\in X$ is such that $\|f\|_X < \e$, the
equation $x = f + H(x, x)$ has a solution in X such that $\|x\|_X \le 2\e$. This solution is
unique in the ball $\bar{B}(0,2\e)$.
\end{lemma}

In light of Duhamel's principle, we define the mild solution of \eqref{eq:pnp} in the following sense:
\begin{definition}
Let $X$ be a Banach space over space and time. We call $\rho\in X$ a mild solution to \eqref{eq:pnp} with initial data $\rho_0$ if $\rho$ satisfies the following equation in $X$:
\begin{equation}\label{eq:mildsol}
    \rho(x,t)=e^{t\Delta}\rho_0+\int_0^te^{(t-s)\Delta}\nabla\cdot\Big(\rho(s,\cdot)\nabla(g*\rho(s,\cdot))\Big)ds.
\end{equation}
\end{definition}

Now we have the following local existence and uniqueness of mild solution:
\begin{proposition}\label{mild:loc}
Suppose $m>\frac{d}{2}$, $d\ge 3$, and the initial data $\rho_0\in L^1(\bbR^d)\cap H^m(\bbR^d)$. Then there exists a $T>0$ such that equation \eqref{eq:pnp} admits a unique mild solution $\rho$ in $C([0,T];L^1(\bbR^d))\cap C([0,T];H^m(\bbR^d))$. If we define $T_b$ to be the largest time of existence, i.e.,
\[
T_b=\sup\{T>0| \text{ \eqref{eq:pnp} has a mild solution in } C([0,T];L^1(\bbR^d))\cap C([0,T];H^m(\bbR^d))\},
\]
then $T_b<\infty$ implies that $\limsup_{t\to T_b^{-}}(\|\rho_t\|_{H^m}+\|\rho_t\|_{L^1})=+\infty$.
Moreover, the integral of the mild solution is preserved, i.e.,
\begin{equation}
	\int_{\bbR^d}\rho(x,t)dx=\int_{\bbR^{d}}\rho_0(x)dx.
\end{equation}
\end{proposition}
\begin{proof}
We will apply Lemma \ref{biler} to prove this result. We set
\[
X:=C([0,T];L^1(\bbR^d))\cap C([0,T];H^m(\bbR^d))
\]
with norm $\|u\|_X:=\|u\|_{C([0,T];L^1)}+\|u\|_{C([0,T];H^m)}$ and define the bilinear form $H$ on $X\times X$ by
\[
H(u,v)=\int_0^te^{(t-s)\Delta}(\nabla\cdot(u(s,\cdot)\nabla(g*v(s,\cdot)))\,ds.
\]
We also denote
\[
\|f\|:=\|f\|_{L^1}+\|f\|_{H^m}
\]
 for $f\in L^1(\bbR^d)\cap H^m(\bbR^d)$.

First, since $H^m$ is an algebra as long as $m>\frac{d}{2}$, for $f_1,f_2\in L^1(\bbR^d))\cap H^m(\bbR^d)$, we have
\begin{equation}\label{hm:algebra1}
    \|f_1\nabla(g*f_2)\|_{H^m}\le C_m\|f_1\|_{H^m}\|\nabla(g*f_2)\|_{H^m}.
\end{equation}
Note that
\begin{equation}\label{hm:algebra2}
    \begin{split}
        \|\nabla(g*f_2)\|_{H^m}^2&=\int_{\bbR^d}|\widehat{\nabla(g*f_2)}(\xi)|^2(1+|\xi|^2)^md\xi\\
        &=\int_{|\xi|\le 1}\frac{|\hat{f}_2|^2(1+|\xi|^2)^m}{|\xi|^2}d\xi+\int_{|\xi|>1}\frac{|\hat{f}_2|^2(1+|\xi|^2)^m}{|\xi|^2}d\xi\\
        &\le C_{m,d}\|\hat{f_2}\|_\infty^2+\|f_2\|_{H^m}^2.
    \end{split}
\end{equation}
Combining \eqref{hm:algebra1}, \eqref{hm:algebra2} and the fact that $\|\hat{f}_2\|_\infty\le\|f_2\|_{L^1}$, one finds that
\begin{equation}\label{hm:algebra3}
\|f_1\nabla(g*f_2)\|_{H^m}\le C_{m,d}\|f_1\|\|f_2\|.
\end{equation}
For $0\le\alpha<1$, one also has the following for $f\in H^m$:
\begin{equation}\label{hm:algebra8}
\begin{split}
\|e^{t\Delta}\nabla f\|_{H^{m+\alpha}}^2 &=\int_{\mathbb{R}^d} |\hat{f}(\xi)|^2(1+|\xi|^2)^{m+\alpha} |\xi|^2 e^{-2|\xi|^2t}d\xi\\
&\le \|f\|_{H^m}^2\sup_{\xi\in\bbR^d}(|\xi|^2(1+|\xi|^2)^\alpha e^{-2|\xi|^2t})\\
&\le C\|f\|_{H^m}^2 (t^{-1}+t^{-1-\alpha}).
\end{split}
\end{equation}

Hence, for $u, v\in X$, one has that
\begin{equation}\label{hm:algebra9}
\sup_{t\in [0, T]}\|H(u,v)\|_{H^{m+\alpha}}\le C(T^{1/2}+T^{(1-\alpha)/2})\|u\|_X\|v\|_X.
\end{equation}

The heat kernel $P(x,t)=\frac{1}{(4\pi t)^\frac{d}{2}}e^{-\frac{|x|^2}{4t}}$ satisfies $\|\nabla P(\cdot, t)\|_{L^1}=\alpha_dt^{-\frac{1}{2}}$, where $\alpha_d$ is a constant. Note that $e^{(t-s)\Delta }\nabla\cdot (u\nabla(g*v))=\int_{\bbR^d} \nabla P(x-y, t-s)\cdot u\nabla(g*v)(y,s)\,dy$. One thus has:
\begin{equation}\label{hm:algebra10}
\begin{split}
\sup_{t\in [0, T]}\|H(u,v)\|_{L^1}&\le C\int_0^t (t-s)^{-\frac{1}{2}} \|u\nabla (g*v)\|_1(s)\,ds,\\
& \le C\|u\|_X\|v\|_X\int_0^T(T-s)^{-1/2}\,ds.
\end{split}
\end{equation}
Note that we used $\|u(s)\nabla (g*v)(s)\|_1\le \|u(s)\|_{2}\|\nabla g*v(s)\|_2
\le C\|u(s)\|_{H^m}\|v(s)\|$ by setting $m=0$ in \eqref{hm:algebra2}.

We now check that $H(u,v)\in X$. By \eqref{hm:algebra9} and \eqref{hm:algebra10}, it is easy to verify that $H(u,v)$ is continuous at $t=0$ in $H^{m+\alpha}$ and $L^1$ norm. We now fix $t>0$.  Pick $\delta_1\in (0, t)$ and set $w=u\nabla(g*v)$. We calculate for $|\delta|$ small enough (note that $\delta$ can be negative) that
\begin{multline}\label{huv:cont1}
\|H(u,v)(t+\delta)-H(u,v)(t)\|_{H^{m+\alpha}}\le\int_0^{t-\delta_1}\|(e^{(t+\delta-s)\Delta}-e^{(t-s)\Delta})\nabla\cdot w\|_{H^{m+\alpha}}ds\\
 +\int_{t-\delta_1}^t\|e^{(t-s)\Delta}\nabla\cdot w(s)\|_{H^{m+\alpha}}ds+\int_{t-\delta_1}^{t+\delta}\|e^{(t+\delta-s)\Delta}\nabla\cdot w(s)\|_{H^{m+\alpha}}ds.
\end{multline}
Using \eqref{hm:algebra8}, the last two terms of \eqref{huv:cont1} are bounded by $C(|\delta|+\delta_1)^{\min(\frac{1}{2},\frac{1-\alpha}{2})}\|u\|_X\|v\|_X$. The first term of \eqref{huv:cont1} is similarly estimated as in \eqref{hm:algebra8}:
\begin{equation}\label{huv:cont2}
\begin{split}
&\int_0^{t-\delta_1}\|(e^{(t+\delta-s)\Delta}-e^{(t-s)\Delta})\nabla\cdot w(s)\|_{H^{m+\alpha}}ds\\
&\le C\|u\|_X\|v\|_X\int_0^{t-\delta_1}\|(1+|\xi|^{2\alpha})(e^{-|\xi|^2(t+\delta-s)}-e^{-|\xi|^2(t-s)})\cdot\xi \|_\infty^{1/2} ds \\
&\le C\|u\|_X\|v\|_X\int_{\delta_1}^t\|\xi e^{-|\xi|^2(s-\delta_1/2)}(1+|\xi|^{2\alpha})(e^{-\delta_1|\xi|^2/2}-e^{-|\xi|^2(\delta_1/2+\delta)})\|_\infty^{1/2} ds.
\end{split}
\end{equation}
By discussing the domains for $|\xi|\ge L$ and $|\xi|\le L$, one can easily find that as $\delta\to 0$, the $\|\cdot\|_{\infty}$ norm goes to zero.  Hence, $H(u, v)$ is continuous at $t$ under $H^{m+\alpha}$ norm. So we have actually verified that $H(u,v)\in C([0,T];H^{m+\alpha}(\bbR^d))$ where $0\le\alpha<1$.

Similar to \eqref{huv:cont1}, we have
\begin{equation}\label{huv:cont3}
\begin{split}
&\|H(u,v)(t+\delta)-H(u,v)(t)\|_{L^1}\le\int_0^{t-\delta_1}\|(e^{(t+\delta-s)\Delta}-e^{(t-s)\Delta})\nabla\cdot w(s)\|_{L^1}ds\\
& +\int_{t-\delta_1}^t\|e^{(t-s)\Delta}\nabla\cdot w(s)\|_{L^1}ds+\int_{t-\delta_1}^{t+\delta}\|e^{(t+\delta-s)\Delta}\nabla\cdot w(s)\|_{L^1}ds.
\end{split}
\end{equation}
Similarly as in \eqref{hm:algebra10}, the last two terms of \eqref{huv:cont3} are controlled by $4C_d\sqrt{|\delta|+\delta_1}\|u\|_X\|v\|_X$. For the first term, we similarly write
\begin{equation}\label{huv:cont4}
\begin{split}
&\int_0^{t-\delta_1}\|(e^{(t+\delta-s)\Delta}-e^{(t-s)\Delta})\nabla\cdot w(s)\|_{L^1}ds\\
&\le C\|u\|_X\|v\|_X \int_0^{t-\delta_1}\|\nabla P(\cdot, t+\delta-s)-\nabla P(\cdot, t-s) \|_{L^1}.
\end{split}
\end{equation}
Since $\nabla P\in C([\delta_1, T], L^1(\mathbb{R}^d))$ and thus uniformly continuous in time on $[\delta_1, T]$. This term goes to zero as $\delta\to 0$. Hence, $H(u, v)\in C([0,T];L^1(\bbR^d))$. We thus have $H(u, v)\in X$
with
\begin{equation}\label{huv:norm}
\|H(u,v)\|_X\le C_{d,m}\sqrt{T}\|u\|_X\|v\|_X.
\end{equation}

Now we apply Lemma \ref{biler} by taking $f= e^{t\Delta}\rho_0$. Since
$$\|e^{t\Delta}\rho_0\|_{L^1}= \|P(\cdot, t)*\rho_0\|_{L^1}\le  \|\rho_0\|_{L^1}$$
$$\|e^{t\Delta}\rho_0\|_{H^m}=\|e^{-|\xi|^2t}(1+|\xi|^2)^{\frac{m}{2}}\hat{\rho}_0\|_{L^2}$$
we find that $\|f\|_X\le \|\rho_0\|_{L^1}+\|\rho_0\|_{H^m}=\|\rho_0\|$. Therefore by Lemma \ref{biler}, equation \eqref{eq:pnp} admits a unique mild solution $\rho\in C([0,T];L^1(\bbR^d))\cap C([0,T];H^m(\bbR^d))$ where $T=\frac{1}{16 C_{d,m}^2}\|f\|_X^{-2}$. Moreover, $\|\rho\|_X\le 2\|f\|_X\le 2\|\rho_0\|$.

Moreover, we claim that the mild solution is also unique on $[0, T_b)$, not just on $[0, T]$. In fact, for two mild solutions $\rho_i(t), i=1,2$. Define $I=\{t: \rho_1(s)=\rho_2(s), \text{for all }s\le [0, t) \}$. Clearly, $I$ is an interval and $[0, T] \subset I$. By viewing $\rho_1(t), t\in I$ as the new initial data and applying Lemma \ref{biler} again, we find that $\rho$ is unique on some interval $[t, t+\epsilon(t)]$ with $\epsilon(t)>0$. Hence, $I$ is an open subinterval of $[0, T_b)$ with the topology inherited from $\mathbb{R}$. Moreover, by the continuity of $\rho_i(t)$, $I$ is also closed. Hence, $I=[0, T_b)$.

If the blow-up criterion does not hold, there exists $M>0$ such that $\sup_{t\in [0, T_b)}\|\rho(t)\|_X\le M$. Set $t_1:=\frac{1}{16 C_{d,m}^2}M^{-2}$. The equation \eqref{eq:pnp} with initial data $\rho(T_b-t_1/2)$ has a mild solution $\tilde{\rho}$ in $C([0,t_1];L^1(\bbR^d))\cap C([0,t_1];H^m(\bbR^d))$. If we define $\bar{\rho}(t)=\rho(t)$ for $0\le t\le T_b-t_1/2$ and $\bar{\rho}(t)=\tilde{\rho}(t-(T_b-t_1/2))$ for $t\in [T_b-t_1/2, T_b+t_1/2]$, then $\bar{\rho}$ is a mild solution on $[0, T_b+t_1/2]$, which contradicts with the definition of $T_b$.

Lastly, we have
\begin{equation}
	\int_{\bbR^d}\rho(x,t)dx=\int_{\bbR^d}\left(e^{t\Delta}\rho_0(x)+\int_0^te^{(t-s)\Delta}\nabla\cdot(\rho(s,\cdot)\nabla(g*\rho(s,\cdot))ds)\right)dx.
\end{equation}
Since we have shown in \eqref{hm:algebra10} that the right side is in $L^1$, we can freely change the order of the integral and the integral preservation follows.
\end{proof}

We now show that the mild solution is a strong solution. We say $\rho\in C([0, T]; L^1(\bbR^d))\cap C([0, T]; H^m(\bbR^d))$ is a strong solution if: (i) $\rho$ is a weak solution that satisfies the equation in the distributional sense; (ii) both $\partial_t\rho$ and $\nabla\cdot(\rho\nabla(g*\rho))+\Delta\rho$ are locally integrable functions on $(0, T)\times \mathbb{R}^d$ so that the equation holds a.e.
\begin{proposition}\label{pro:strongunique}
Let $\rho_0\in L^1(\bbR^d)\cap H^m(\bbR^d)$ with $m>\frac{d}{2}$. Then the mild solution $\rho$ is a strong solution belonging to $C^\infty((0,T_b),H^{m'}(\bbR^d))$ for any $m'\ge m$. Moreover, the strong solution is unique.
\end{proposition}
\begin{proof}
We take $T\in (0, T_b)$. From the proof of previous proposition,  for $0\le\alpha<1$,
\[
H(\rho,\rho)\in C([0,T];H^{m+\alpha}(\bbR^d)).
\]
Meanwhile, since $\rho_0\in H^m(\bbR^d)$, it is easy to verify that
$e^{t\Delta}\rho_0\in C((0,T]; H^{m'})$ for any $m'>0$. Therefore we see that $\rho$ is in $C((0,T];H^{m+\alpha}(\bbR^d))$.
	
Now for any $0<t_1<T$, we take $\alpha=\frac{1}{2}$ with the new initial value $\rho_0^{(1)}=\rho_{\frac{t_1}{2}}$. Then $\rho^{(1)}(t):=\rho(t-\frac{t_1}{2})$ is a mild solution of \eqref{eq:pnp} in $C([0,T-\frac{t_1}{2}];H^{m+\alpha}(\bbR^d))\cap C([0,T-\frac{t_1}{2}];L^1(\bbR^d))$. Therefore the previous argument implies that $\rho^{(1)}\in C((0,T-\frac{t_1}{2}];H^{m+2\alpha}(\bbR^d))$. Then we can take the new initial value $\rho_0^{(2)}=u_{\frac{t_1}{4}}^{(1)}$ along with $\rho^{(2)}(t)=u^{(1)}(t-\frac{t_1}{4})$. Iterating this process for $2(m'-m)+2$ times, we find that $\rho\in C([t_1,T],H^{m'}(\bbR^d))$.

Take $t_1>0$.  Let $\bar{\rho}(t)=\rho(t+t_1)$. Then, $\bar{\rho}$ satisfies
\[
\bar{\rho}(t)=e^{t\Delta}\bar{\rho}(0)+\int_{0}^te^{(t-s)\Delta}\nabla\cdot(\bar{\rho}\nabla(g*\bar{\rho})).
\]
We have $w:=\nabla\cdot(\bar{\rho}\nabla(g*\bar{\rho})) \in C([0,T-t_1];H^{m}(\bbR^d))$ for any $m>0$. It then follows
\begin{equation}
\Delta \bar{\rho}(t)=\Delta e^{t\Delta}\bar{\rho}(0)+ \int_0^te^{(t-s)\Delta}\Delta w_s\,ds.
\end{equation}
By the property for heat equation with $L^2$ initial data, we have $e^{t\Delta }u-u=\int_0^t \Delta e^{\tau\Delta }u \,d\tau$ if $u\in L^2$. Hence,
\begin{equation}
	\begin{split}
	\int_0^{t} (\Delta \bar{\rho}(\tau)+w(\tau))d\tau&=\int_0^{t}\Delta e^{\tau\Delta}\bar{\rho}_0d\tau+ \int_0^{t}\int_0^{\tau}e^{(\tau-s)\Delta}\Delta w_s\,dsd\tau+\int_0^{t}w(\tau)d\tau\\
	&=(e^{t\Delta}\bar{\rho}_0-\bar{\rho}_0)+\int_0^{t}\int_s^{t}e^{(\tau-s)\Delta}\Delta w_s\,d\tau ds+\int_0^{t}w(\tau)\,d\tau\\
	&=\bar{\rho}(t)-\bar{\rho}_0.
	\end{split}
\end{equation}
We exchanged the order of integral since $e^{(\tau-s)\Delta}\Delta w_s$ is bounded under $L^2$ norm.  This identity first of all implies that $\bar{\rho}$ is a weak solution since $\rho\in C([0, T]; L^2)$. Moreover, it also implies that $\bar{\rho}\in C^{\infty}([t_1, T])$ under any $H^m$ norm.  Hence, taking derivative on time, we find that $\bar{\rho}$ is a strong solution. Since $t_1$ is arbitrary, the claim follows.

The strong solution is a mild solution on $[0, T]$. The uniqueness then  follows trivially by the uniqueness of mild solutions.
\end{proof}

We are more interested in the non-negative initial data due to the problem we consider.
\begin{proposition}\label{pro:nonblowup}
Besides the conditions in Proposition \ref{mild:loc}, if we also have $\rho_0\ge 0$, then
\begin{enumerate}
\item For all $t$ in the integral of existence, we have $\rho(x,t)\ge 0$.

\item The strong solution exists globally, i.e. $T_b=\infty$.
\end{enumerate}
\end{proposition}

\begin{proof}

1. The proof of non-negativity follows in a similar way as in \cite{Lizhen2018}. Here, we sketch the proof briefly.

We fix an arbitrary $T\in (0, T_b)$ and let $\rho$ be the mild solution on $[0, T]$.
We consider the approximated problem
\[
\rho_n(t)=e^{t\Delta}\rho_0+\int_0^t e^{(t-s)\Delta}\nabla\cdot(\rho_n^+(s)\nabla(g^{\e_n}*\rho)),
\]
where $\e_n=1/n$. Since $\nabla(g^{\e_n}*\rho)$ is a smooth function with the derivatives bounded, then $\rho_n\in C([0, T], H^1)\cap C([0, T], L^1) \cap C^{\infty}((0, T), H^1)$.  Then, for $t>0$, it holds in $H^{-1}$ that
\[
\partial_t\rho_n=\nabla\cdot(\rho_n^+(s)\nabla(g^{\e_n}*\rho)+\Delta \rho_n.
\]
Multiply $\rho_n^-=-\min(\rho_n, 0)$ on both sides and integrate.
The right hand is equal to $\|\nabla \rho_n^-\|_2^2$.
Consider the left hand side. Since  $\partial_t\rho_n=\lim_{h\to 0^+}\frac{\rho_n(t)-\rho_n(t-h)}{h}$ converges in $L^2$, we have
\[
\langle \rho_n^-, \partial_t\rho_n\rangle=\lim_{h\to 0^+}\left\langle \rho_n^-,
\frac{\rho_n(t)-\rho_n(t-h)}{h} \right\rangle
\le -\frac{1}{2}\lim_{h\to 0^+}\frac{\|\rho_n^-(t)\|_2^2-\|\rho_n^-(t-h)\|^2}{h}.
\]
Since $|\rho_n^-(t_2)-\rho_n^-(t_1)|\le |\rho_n(t_2)-\rho_n(t_1)|$ holds pointwise and therefore in $L^2$, we find that $t\to \|\rho_n^-\|_2$ is in $C[0, T]\cap C^1(0, T)$ and
\[
\partial_t\|\rho_n^-\|_2^2\le 0.
\]
This implies that $\rho_n^-=0$ and thus $\rho_n^+=\rho_n$.
As $n\to\infty$, we can show that $\rho_n\to \rho$ in $C([0, T], L^2)$, which further implies that $\rho\ge 0$ on $[0, T]$.

2.  It suffices to show that the solution does not blow up in $L^1$ and $H^m$ norm in a finite time. By the integral preservation and positivity preservation, $\|\rho(t)\|_{L^1}=\|\rho_0\|_{L^1}$. Hence, we only need to consider $H^m$ norm.

Note $\rho_0\in L^p$ for all $p\in [1,\infty)$. Using the facts that $\rho\in H^{m'}$ for any $m'>0$ and that $\rho$ is smooth in time for $t>0$, we can multiply the equation with $p\rho^{p-1}$ and integrate to have for $t>0$,
\begin{multline}\label{nonblowup1}
\partial_t\|\rho\|_p^p=-p(p-1)\inner{\nabla\rho}{\rho^{p-2}\nabla\rho}-\inner{\rho^2}{p\rho^{p-1}}+\inner{\nabla(g*\rho)}{\nabla\rho^{p}}\\
=-(p-1)\int_{\mathbb{R}^d} (p\rho^{p-2}|\nabla\rho|^2+\rho^{p+1})\,dx.
\end{multline}
Using the non-negativity of $\rho$, we find $\|\rho\|_p\le C_p$ is uniformly bounded.

Now, we consider $n=[m]+i$ (where $i=1,2$ so that $n$ is even) and use the data at some $t_1>0$ as the initial data. Set $\|f\|_{\dot{H}^s}=\|(-\Delta)^{\frac{s}{2}}f\|_{L^2}=\||\xi|^s\hat{f}\|_{L^2}$,  then $\|\rho\|_{H^n}$ can be controlled by $\|\rho\|_{\dot{H}^n}$ and $\|\rho\|_{L^2}$. Therefore we only need to show $\|\rho\|_{\dot{H}^n}$ does not blow up, which clearly will indicate that the original $\|\cdot\|_{H^m}$ norm does not blow up.
Multiply \eqref{eq:pnp} by $(-\Delta)^n\rho$ and integrate, we have for $t>t_1$
	\begin{equation}\label{nonblowup2}
	\frac{1}{2}\partial_t\|\rho\|_{\dot{H}^n}^2=-\|\rho\|_{\dot{H}^{n+1}}^2-\inner{\rho^2}{(-\Delta)^n\rho}+\inner{\nabla\rho\cdot\nabla(g*\rho)}{(-\Delta)^n\rho}.
	\end{equation}
	For the second term of \eqref{nonblowup2}, after integrating by parts for $n$ times, we obtain
\begin{equation}
	-\inner{\rho^2}{(-\Delta)^n\rho}=-\langle (-\Delta)^{n/2}\rho, (-\Delta)^{n/2}(\rho^2)\rangle.
\end{equation}
Expanding $(-\Delta)^{n/2}(\rho^2)$ out, this contains terms of the form $C_{\ell}D^{\ell}\rho D^{n-\ell}\rho$ where $\ell=0, 1,\ldots, n$ and $D$ denotes any partial derivative. For the $\ell=0, n$ terms, we use the nonnegativity of $\rho$ and find $-\int_{\bbR^d} \rho |(-\Delta)^{n/2}\rho|^2\,dx\le 0$. Consider that $1\le \ell\le n-1$. By Gagliardo-Nirenberg inequality:
\begin{equation}\label{g-nineq}
	\|D^jf\|_{L^p}\le C\|f\|_{L^r}^{1-\alpha}\|f\|_{\dot{H}^{n+1}}^\alpha, \ \alpha=\frac{j-d/p+d/r}{(n+1)-d/2+d/r}.
	\end{equation}
	Setting $p_{\ell}=\frac{2n}{\ell}, q_{\ell}=\frac{2n}{n-\ell}$, applying H\"{o}lder inequality and \eqref{g-nineq}, we find that for $1\le \ell\le n-1$,
	\begin{equation}\label{nonblowup3}
	|\inner{D^\ell\rho D^{n-\ell}\rho}{ D^{n}\rho}|\le\|D^\ell\rho\|_{L^{p_{\ell}}}\|D^{n-\ell}\rho\|_{L^{q_{\ell}}}\|\rho\|_{\dot{H}^n}\le C\|\rho\|_{\dot{H}^{n+1}}^{\frac{n-d/2+2d/r}{(n+1)-d/2+d/r}}\|\rho\|_{\dot{H}^n}.
	\end{equation}
Here we have used the fact that $\|\rho_t\|_{L^r}\le\|\rho_0\|_{L^r}$.
Here, $p_{\ell}$ and $q_{\ell}$ are chosen so that the corresponding $\alpha\in (0, 1)$. We pick $r>d$ and then the power of $\|\rho\|_{\dot{H}^{m+1}}$ is less than $1$.

For the third term of \eqref{nonblowup2}, we similarly have
\begin{equation}\label{nonblowup4}
	\inner{\nabla\rho\cdot\nabla(g*\rho)}{(-\Delta)^n\rho}=\inner{(-\Delta)^{\frac{n}{2}}(\nabla\rho\cdot\nabla(g*\rho))}{(-\Delta)^{\frac{n}{2}}\rho}.
\end{equation}
Expanding out, we have terms of the form $(\nabla D^{n-\ell}\rho) D^{\ell}\nabla(g*\rho)$. The $\ell=0$ term contributes to
\[
\int_{\bbR^d}\nabla ((-\Delta)^{n/2}\rho)\cdot \nabla(g*\rho) (-\Delta)^{\frac{n}{2}}\rho\,dx=-\frac{1}{2}\int_{\mathbb{R}^d} \rho |(-\Delta)^{\frac{n}{2}}\rho|^2\,dx\le 0.
\]
When $\ell\ge 1$, by the singular integral theory, we have
\begin{gather}\label{eq:auxSingularInt}
\|D^{\ell}\nabla(g*\rho)\|_{L^{p_{\ell}}}=\|D\nabla(g*D^{\ell-1}\rho)\|_{L^{p_{\ell}}}
\le \|D^{\ell-1}\rho\|_{L^{p_{\ell}}},~1<p_{\ell}<\infty.
\end{gather}
Due to this reason, we find that when $\ell=1$, the pairing is controlled by 
\[
\begin{split}
\langle (\nabla D^{n-\ell}\rho) D^{\ell}\nabla(g*\rho), (-\Delta)^{\frac{n}{2}}\rho \rangle 
&\le 
\|\nabla D^{n-1}\rho\|_{2+\delta} \|D\nabla(g*\rho)\|_{2(2+\delta)/\delta}  \|\rho\|_{\dot{H}^n} \\
&\le C\|\nabla D^{n-1}\rho\|_{2+\delta} \|\rho\|_{\dot{H}^n}
\end{split}
\]
By Galiardo-Nirenberg inequality again, $\|\nabla D^{n-1}\rho\|_{2+\delta}\le C\|\rho\|_{\dot{H}^{n+1}}^{\alpha}
\|\rho\|_{\dot{H}^{n}}^{1-\alpha}$, where $\alpha=\frac{d\delta}{2(2+\delta)}$.
For $\ell>1$, using \eqref{eq:auxSingularInt}, the pairing $\langle (\nabla D^{n-\ell}\rho) D^{\ell}\nabla(g*\rho), (-\Delta)^{\frac{n}{2}}\rho \rangle$ is similarly controlled as in \eqref{nonblowup3}. Hence we finally have for $t>t_1$
\begin{equation}
	\frac{1}{2}\partial_t\|\rho\|_{\dot{H}^n}^2\le
	-\|\rho\|_{\dot{H}^{n+1}}^2
	+C\|\rho\|_{\dot{H}^n}^{2-\alpha}\|\rho\|_{\dot{H}^{n+1}}^{\alpha}+C\|\rho\|_{\dot{H}^n}\|\rho\|_{\dot{H}^{n+1}}^{\nu}
\end{equation}
where $\alpha\in(0, 1)$ and $\nu\in (0, 1)$. This gives that $\|\rho\|_{\dot{H}^n}$ never blows up in finite time for $t>t_1$, which further implies that $\|\rho\|_{H^m}$ does not blow up.
\end{proof}

\section{The missing proofs}\label{app:missingpf}

\begin{proof}[Proof of Proposition \ref{tightI}]

Note that for any $N$, $0\le s<t\le T$, one has
\[
X_t^{1,N}-X_s^{1,N}=\frac{1}{N}\int_s^t \sum_{j\neq 1}^N
F(X_r^{1,N}-X_r^{j,N})\,dr+\sqrt{2}(B_t^1-B_s^1).
\]
This then motivates us to define
\begin{gather*}
Z_N:=\sup_{s,t:s<t}\frac{\sqrt{2}|B_t^1-B_s^1|}{(t-s)^{1/2}},~~
U_N:=\frac{1}{N}\left(\int_0^T\left(\sum_{j\neq 1}^N F(X_t^1-X_t^j)\right)^2dt\right)^{1/2}.
\end{gather*}
Clearly,
\[
|X_t^{1,N}-X_s^{1,N}|\le (t-s)^{1/2}(Z_N+U_N).
\]

Moreover, $Z_N$'s have the same distribution for all $N$, and
\[
Z_N<\infty, a.s.
\]
It follows that
\[
\lim_{R_1\to\infty}\sup_{N\ge 2}\mathbb{P}(|Z_N|>R_1)=0.
\]

Using the energy estimate \eqref{energy:est2}, we have
\[
\mathbb{E} U_N^2 \le \mathcal{E}(\rho_0).
\]
Moreover, $\mathbb{E}|X_0^{1,N}|^2<\infty$.

We define
\[
K:=\{ X\in C([0, T]; \bbR^d), |X_0|\le A, |X_t-X_s|\le R(t-s)^{1/2}, \forall 0\le s<t\le T \}.
\]
Clearly, $K$ is a compact set in $C([0, T],\bbR^d)$ by Arzela-Ascoli theorem.

Moreover,
\begin{multline*}
\sup_{N\ge 2}\mathbb{P}(X_t^{1,N}\notin K) \le \sup_{N\ge 2}\mathbb{P}(|X_0^{1,N}|>A)
+\sup_{N\ge 2}\mathbb{P}(Z_N+U_N>R)\\
\le \sup_{N\ge 2}\mathbb{P}(|X_0^{1,N}|>A)
+\sup_{N\ge 2}(\mathbb{P}(Z_N>R/2)+\mathbb{P}(U_N>R/2)).
\end{multline*}
 Using the uniform bound on the moments of $X_0^{1,N}$, $U_N$, we find that
 for any $\e>0$, there exist $A>0, R>0$ such that
 \[
 \sup_{N\ge 2}\mathbb{P}(X_t^{1,N}\notin K)<\epsilon,
 \]
 which concludes the tightness of the law of $X^{1,N}$.

The tightness of $\mathcal{L}(\mu^N)$ follows from (i) and the exchangeability of the system. See \cite[Proposition 2.2]{sznitman1991}.

\end{proof}

\begin{proof}[Proof of Proposition \ref{pro:improvedweak}]

Recall that $h=g*\rho$. By the assumption, one has
\[
\rho\in L^{\infty}((0, T); H^{-1}), ~~\nabla h\in L^{\infty}((0, T); L^2).
\]
Since $\int_{\bbR^3} \rho^{3/2}\,dx\le \|\rho\|_{H^{-1}}\|\sqrt{\rho}\|_{\dot{H}^1}$, we thus have
\[
\|\rho\|_{L^3((0, T); L^{3/2})}\le C\int_0^T\|\nabla \sqrt{\rho}\|_{L^2}^2\,dt
= C\int_0^TI(\rho)\,dt.
\]
Interpolating this with $\rho\in L^1(0, T; L^3)$, we have $\rho\in L^{3r/(5r-6)}(0, T; L^r)$ for $r\in [3/2, 3]$.

Moreover, $\nabla^2h\in L^1(0, T; L^3)$ since $\rho\in L^1(0, T; L^3)$, interpolating with $\nabla h\in L^{\infty}(0, T; L^2)$ by the Gagliardo-Nirenberg inequality, we find that
\[
\nabla h\in L^{q/(q-2)}((0, T); L^q), q\ge 2.
\]
Hence, $\rho\nabla h\in L^{3p/(8p-6)}(0, T; L^p)$ for $p\in [1, 6/5]$. In particular, we have
\[
\rho \nabla h\in L^{6/5}((0, T), L^{12/11}).
\]
Since $\Delta$ is bounded from $W^{1,q}$ to $W^{-1,q}$, by Lemma \ref{limitmeasure:lp}, we have
$\Delta\rho\in L^{8/5}((0, T), W^{-1, 12/11})$.

Recall that
\begin{equation}\label{eq:improvedweak1}
    \langle\rho_t,\phi\rangle-\langle \rho_0,\phi\rangle-\int_0^t\int_{\bbR^{3}}\nabla\phi(x)\cdot \nabla h(x)\rho_s(x)\,dxds-\int_0^t\langle\rho_s,\Delta\phi\rangle ds=0
\end{equation}
for any $\phi\in C_c^2(\bbR^3)$ and $t\in (0, T]$. Using the regularity, we find that this holds for all $\phi\in C_b^2(\bbR^3)$. In fact, we can take smooth truncation of $\phi_n=\phi\chi_n$ where $\chi_n=\chi(x/n)$ and $\chi=1$ in $B(0, 1)$. Then,
$\nabla (\phi\chi_n)\cdot\nabla h\to \nabla\phi\cdot \nabla h$ in $L^2(0, T; L^4)$ and $\langle \rho_s, D^{\alpha}(\phi_n)\rangle\to \langle \rho_s, D^{\alpha}\phi\rangle$ where $|\alpha|\le 2$. The latter holds because $\rho_s\in L^1$ and the
$D^{\alpha}(\phi-\phi_n)$ is bounded and nonzero only outside $B(0, n)$. That $\int_0^t \langle \rho_s, \Delta\phi_n\rangle\,ds
\to \int_0^t \langle \rho_s, \Delta\phi\rangle\,ds$ holds by dominate convergence theorem.

Then, we claim that for any $\phi\in C^1([0, T], C_b^2(\bbR^3))$ and $t\in (0, T]$, it holds that
\begin{multline}\label{weaksol2}
\langle \rho_t, \phi(x, t)\rangle-\langle \rho_0,\phi(x,0)\rangle-\int_0^t\int_{\bbR^{3}}\rho(x,s)\partial_t\phi \,dxds\\
+\int_0^t\int_{\bbR^{3}}\nabla\phi(x,s)\cdot\nabla h(x)\rho(x,s)\,dxds-\int_0^t\langle\rho(x,s),\Delta\phi(\cdot, s)\rangle ds=0.
\end{multline}
In fact, we can take $t=t_{1}$ and $t=t_2$ in \eqref{eq:improvedweak1} and take the difference to obtain
$\langle \rho_{t_2}, \varphi\rangle-\langle \rho_{t_1}, \varphi\rangle
-\int_{t_1}^{t_2}\int \ldots-\int_{t_1}^{t_2}\ldots=0$, where the omitted content is clear.
Then, we can take $\varphi=\phi(\cdot, t_n)$ so that we have kind of Riemann sum. The regularity ensures that the Riemann sum converges to the desired integral form.

For $\phi\in C^1([0, T]; C_c^2(\bbR^3))$, we then have
\[
\int_0^t\langle \partial_t\rho, \phi\rangle
=\int_0^t\int_{\bbR^{3}}\nabla\phi(x,s)\cdot\nabla h(x)\rho(x,s)\,dxds
+\int_0^t\langle \nabla\rho, \nabla\phi\rangle\,ds,
\]
where $\partial_t\rho$ is the distributional derivative of $\rho$. Clearly, the right hand side
is a bounded functional for $\phi\in L^6(0, T; W^{1, 12})$. By possible mollification procedure, we find
\begin{gather}
\partial_t\rho=\nabla\cdot(\rho\nabla h)+\Delta\rho,~\text{in}~L^{6/5}((0, T), W^{-1, 12/11}).
\end{gather}

In fact, this weak solution is also a mild solution. To see this, we mollify $\rho$ as
\[
\rho^{\e,\delta}=J_1^{\delta}(t)*\rho*J_2^{\e}(x)
\]
Here, $J_1$ is the mollification in time while $J_2$ is in space. Then, on $t\in (\delta, T-\delta)$, we have
\[
\partial_t\rho^{\e,\delta}=J_1^{\delta}*J_2^{\e}*\partial_t\rho
=\nabla\cdot(J_1^{\delta}*\rho\nabla h*J_2^{\e})+\Delta \rho^{\e,\delta}.
\]
Since all functions are smooth and bounded, with derivatives bounded,  we have for $\delta<t_1<t<T-\delta$ that
\[
\rho^{\e,\delta}(t)=e^{(t-t_1)\Delta}\rho^{\e,\delta}(t_1)
+\int_{t_1}^t e^{(t-s)\Delta}\nabla\cdot(J_1^{\delta}*\rho\nabla h*J_2^{\e})\,ds
\]

Now, we claim that
\[
\rho^{\e}(t):=J_2^{\e}*\rho(t)
\]
is a bounded continuous function on $[0, T]\times \bbR^3$.
In fact, we let $\phi_x(y)=J_2^{\e}(x-y)\in C_b$.
Then, we can define
 \[
 p_{x,t}(X)=\phi_x(X(t)).
 \]
 This is a continuous bounded functional on $C([0, T]; \mathbb{R}^3)$.
 Then, $p_{x,t}\to p_{(x,t_0)}$ pointwise as $t\to t_0$ and are bounded functional. Hence, we have by dominate convergence theorem
 \[
 \int_{C([0, T]; \mathbb{R}^3)}p_{(x,t)}(X) d\mu\to  \int_{C([0, T]; \mathbb{R}^3)}p_{(x, t_0)}(X) d\mu.
 \]
 This means
 \[
 \rho^{\e}(x, t)\to \rho^{\e}(x, t_0),~\forall x\in \bbR^3.
 \]
 Hence, $\rho^{\e}(x,t)$ is continuous and bounded. Since $\rho^{\e}\in L^1\cap L^{\infty}$, then taking $\delta\to 0$, we find
 \[
 \rho^{\e,\delta}(t)\to \rho^{\e}(t),\text{in }L^1\cap L^{3/2}, t\in (0, T)
 \]

Note that
\[
\|\nabla P\|_r\le C_rt^{3/(2r)-2},~r\in [1,\infty].
\]
Picking $r=4/3$, applying \cite[Theorem 4]{hardy1928} with $\rho\nabla h\in L^{6/5}(0, T; L^{12/11})$, we find that $\int_{t_1}^t \nabla e^{(t-s)\Delta}\cdot (\rho\nabla h)\,ds\in L^{4/3}((0, T); L^{3/2})$. Taking $\delta\to 0$, we have in $L^{4/3}((0, T); L^{3/2})$ that
\[
\int_{t_1}^t e^{(t-s)\Delta}\nabla\cdot(J_1^{\delta}*\rho\nabla h*J_2^{\e})\,ds
\to \int_{t_1}^t e^{(t-s)\Delta}\nabla\cdot(J_2^{\e}*(\rho \nabla h))\,ds.
\]
Then, we have in $L^{4/3}((0, T); L^{3/2})$ that
\[
\rho^{\e}(t)=e^{(t-t_1)\Delta}\rho^{\e}(t_1)
+\int_{t_1}^t e^{(t-s)\Delta}\nabla\cdot(J_2^{\e}*(\rho\nabla h))\,ds.
\]

Now, for any $t>0$, we take $t_1\to 0$. By dominate convergence, we have:
\[
\rho^{\e}(t)=e^{t\Delta}\rho^{\e}(0)
+\int_{0}^t e^{(t-s)\Delta}\nabla\cdot(J_2^{\e}*(\rho\nabla h))\,ds.
\]
Eventually, we take $\e\to 0$ and we have in $L^{4/3}((0, T); L^{3/2})$ that
\begin{gather}
\rho(t)=e^{t\Delta}\rho_0
+\int_{0}^t e^{(t-s)\Delta}\nabla\cdot(\rho\nabla h)\,ds.
\end{gather}
\end{proof}

\begin{proof}[Proof of Proposition \ref{prop:uniquenessresult}]

{\bf Step 1} The $L^p$ bound for $t>0$.

We fix $\delta>0$ and then
\[
\rho\in L^2([\delta, T], L^2(\bbR^3)).
\]

Mollifying the equation for $\rho$ in Proposition \ref{pro:improvedweak}, we have
\begin{equation}\label{weakmolify}
\partial_t\rho^{\e}-\nabla h\cdot\nabla \rho^{\e}-\Delta\rho^{\e}
=J^{\e}*(\nabla\cdot(\rho\nabla h))-\nabla h\cdot\nabla \rho^{\e}=:r^{\e}.
\end{equation}
Note that we do not have $\nabla\rho\cdot\nabla h\in L^1(\delta, T; L^1)$, so $\nabla\rho\cdot\nabla h$ may not be a distribution. This is why we cannot have such a term in the equation. Recall $\rho\in L^2(\delta, T; L^2(\bbR^3))$, and $\nabla^2h\in L^2(\delta, T; L^2(\bbR^3))$ by singular integral theory.  We then use the proof of Lemma II.1 in \cite{diperna1989} and conclude
\[
r^{\e}\to -\rho^2, ~L^1(\delta, T; L_{\loc}^1(\bbR^3)).
\]
(There is a small typo in the proof of Lemma II.1 in \cite{diperna1989}, where $\e^{-N}$ is lost in the expressions on P.517). In fact,
\begin{multline*}
J^{\e}*(\nabla\cdot(\rho\nabla h))-\nabla h\cdot\nabla \rho^{\e}
=-\int_{\bbR^3} \nabla_y(J^{\e}(x-y))\cdot \rho(y)\nabla h(y)\,dy-\nabla h\cdot\nabla \rho^{\e}\\
=\int_{\bbR^3}(\nabla J^{\e})(x-y)\cdot(\nabla h(y)-\nabla h(x))\rho(y)\,dy.
\end{multline*}
By our construction, $\|\nabla J^{\e}\|_{\infty}\le C\e^{-4}$. It follows that
\begin{multline*}
\Big\|\int_{\bbR^3}(\nabla J^{\e})(x-y)\cdot(\nabla h(y)-\nabla h(x))\rho(y)\,dy \Big\|_{L^1(B(0, R))} \\
\le C\int_{B(0, R)}dx \int_{|y-x|\le C\e}\e^{-3}\rho(y)\frac{|\nabla h(y)-\nabla h(x)|}{\e}\,dy.
\end{multline*}
Let $D_1=\{(x, y): x\in B(0, R), |y-x|\le C\e\}$. We then have the above controlled by
\[
\left(\iint_{D_1}\rho^2(y)dxdy \right)^{1/2}\e^{-3}\left(\iint_{D_1}\Big(\frac{|\nabla h(y)-\nabla h(x)|}{\e}\Big)^2\,dxdy\right)^{1/2}.
\] 
The first term is controlled by $C\e^{3/2}\|\rho\|_{L^2(B(0, R+1))}$. The second term is controlled by
\begin{multline*}
\e^{-3}\left(\int_{B(0, R)}dx \int_{|y-x|\le C\e}dy\Big(\frac{|\nabla h(y)-\nabla h(x)|}{\e} \Big)^2 \right)^{1/2} \\
\le \e^{-3/2}\left(\int_{B(0, R)}dx \int_{|z|\le C}dz \Big(\int_0^1|\nabla^2h(x+t\e z)|\Big)^2 \right)^{1/2}.
\end{multline*}
Since
\[
\left[\int_{B(0, R)}dx(\int_0^1|\nabla^2h(x+t\e z)|)^2 \right]^{1/2}
\le \int_0^1\|\nabla^2h(\cdot+t\e z)\|_{L^2(B(0, R))}\,dt
\le \|\nabla^2 h\|_{L^2(B(0, R+1))}.
\]
We have the second term controlled by $\e^{-3/2}C\|\nabla^2 h\|_{L^2(B(0, R+1))}$. Hence, 
\[
\|r^{\e}(t)\|_{L^1(B(0, R))}\le C\|\rho(t)\|_{L^2(B(0, R+1))}\|\nabla^2h(t)\|_{L^2(B(0, R+1))}.
\]
This bound is uniform in $\e$. With the time dimension added in, the corresponding norms are similarly controlled.
By a density argument, we can then approximate $\rho$ and $\nabla^2h$ with smooth functions in their respective spaces. For smooth functions, the limit is clearly
$\rho\Delta h=-\rho^2$. 

Recall that $\rho\,dx\in C([0, T]; C_b(\mathbb{R}^3)')$, we have $\rho^{\e}\in C([0, T]; L^1(\mathbb{R}^3))$.  Using basically the same argument as in Step 1 of the proof of Lemma 2.5 in \cite{egana2016}, we obtain that $\rho^{\e}$ is a Cauchy sequence in $L^1(\delta, T; L_{\loc}^1)$ as $\epsilon\to 0$. In fact, for convex function $\beta\in C^1(\bbR^3)$, we have the following chain rule:
\begin{equation}\label{mollifychain}
\partial_t\beta(\rho^\epsilon)=\nabla\beta(\rho^\epsilon)\cdot\nabla h+\beta'(\rho^{\e})r^\e+\Delta\beta(\rho^\e)-\beta''(\rho^\e)|\nabla\rho^\e|^2.
\end{equation}
Equation \eqref{weakmolify} and \eqref{mollifychain} will hold if we replace $\rho^{\e}$ with $\rho^{\e_1}-\rho^{\e_2}$ and $r^{\e}$ with $r^{\e_1}-r^{\e_2}$ since \eqref{weakmolify} is linear. In particular, we choose $\beta(s)=s^2/2$ for $|s|\le A$ and $\beta(s)=A|s|-A^2/2$ for $|s|\ge A$. This will give $\lim_{\e_1\to 0, \e_2\to 0}\sup_{0\le t\le T}\int_{\bbR^3} \beta(\rho^{\e_1}-\rho^{\e_2})\chi(x)\,dx=0$ for any $\chi\in C_c^{\infty}$. (There is only one difference from  \cite{egana2016}\tcb{:} to justify $\int_{\bbR^3} \beta(\rho^{\e_1}-\rho^{\e_2})\nabla h \cdot \nabla\chi \to 0$, we use $\nabla h\in L^{\infty}(L^2)$ and $\rho^{\e_1}-\rho^{\e_2}\to 0$ in $L^2(\delta, T; L^2)$).

Consequently, we have $\rho\in C([\delta, T]; L_{\loc}^1(\bbR^3))$. Moreover, using the uniform of estimates $\int_{\bbR^3} \rho |x|^2\,dx$ so that $\{\rho(t)\}_{t\in [\delta, T]}$ is tight, we have
\begin{gather}
\rho\in C([\delta, T]; L^1(\bbR^3)).
\end{gather}

Next, following the proof of Lemma 2.5 in \cite{egana2016}, taking some non-negative test function $\chi\in C_c^{\infty}(\bbR^3)$ and integrate \eqref{mollifychain} over $[\delta_1, t_1]$ where $\delta\le \delta_1\le t_1$, we find
\begin{multline}
\int_{\bbR^3} \beta(\rho_{t_1}^\epsilon)\chi dx + \int_{\delta_1}^{t_1}\int_{\bbR^3}\beta''(\rho_{s}^\epsilon)|\nabla \rho_{s}^\epsilon|^2\chi dxds
= \int_{\bbR^3} \beta(\rho_{\delta_1}^\epsilon)\chi dx \\
+ \int_{\delta_1}^{t_1}\int_{\bbR^3}\{\beta(\rho_s^\e)\Delta\chi+(\beta'(\rho_s^\e)r^\e+\rho_s\beta(\rho_s^\e))\chi-\beta(\rho_s^\e)\nabla h\cdot\nabla\chi\}dxds.
\end{multline}
We first consider $\beta\in C^2(0,\infty)$ that is (i) convex, linear outside a compact set ( i.e., $\beta''$ is continuous with compact support); (ii) for any $|u|\le L$, there is $C(L)$ such that $|\beta(u)|\le C(L)|u|$. Taking the limit $\e\to0$ first and then $\chi_R(x)=\chi(\frac{x}{R})$ with $R\to\infty$,  we obtain that
\begin{multline}\label{convexineq}
\int_{\bbR^3} \beta(\rho_{t_1}) dx+ \int_{\delta_1}^{t_1}\int_{\bbR^3}\beta''(\rho_{s})|\nabla \rho_{s}|^2 dxds \\
\le \int_{\bbR^3}\beta(\rho_{\delta_1})dx+\int_{\delta_1}^{t_1}\int_{\bbR^3}(-\beta'(\rho_s)\rho_s^2+\rho_s\beta(\rho_s))^+ dxds.
\end{multline}
The left hand side is obtained by Fatou's lemma since $\beta''(s)\ge0$. The convergence of the first term on the right can be obtained by decomposing $\beta(u)=(\beta(u)-A|s|)+A|s|$. $(\beta(u)-A|s|)$ is treated by dominate convergence theorem while $A|s|$ is treated by the $L^1$ convergence of mollification. Other terms are dealt with by the regularity and the convergence of $r^{\e}\to -\rho^2$ in $L^1(\delta, T; L_{\loc}^1(\bbR^3))$.

For general convex $\beta\in C^2(0,\infty)$ with  (i)  $0\le\beta(u)\le C(1+u|\log(u)|)$, $u\beta(u)-u^2\beta'(u)\le C(1+u^2)$; (ii) for any $|u|\le L$, there is $C(L)$ such that $|\beta(u)|\le C(L)|u|$, we may choose a sequence of smooth convex functions $\beta_R$ with linear growth at infinity to approximate $\beta$ and obtain \eqref{convexineq}.
In fact, for such $\beta$, $\int_{\delta_1}^{t_1}\int_{\bbR^3}(-\beta'(\rho_s)\rho_s^2+\rho_s\beta(\rho_s))^+ dxds$ is integrable by decomposing the integrals into domains for $\rho\le 1$ and $\rho\ge 1$.

 Now mimicking the proof of Lemma 2.7 in \cite{egana2016}, we take for $p\ge 2$
\[
\beta(u) := \frac{u^p}{p}\textbf{1}_{u\le K}+\left(\frac{K^{p-1}}{\log K}(u\log u-u)-\frac{1}{q}K^p+\frac{K^{p}}{\log K}\right)\textbf{1}_{u\ge K},
\]
where $\frac{1}{p}+\frac{1}{q}=1$. Then $\beta$ is convex, nonnegative and $\beta'(x)\ge0$. Moreover, $u\beta(u)-u^2\beta'(u)\le 0$ if $K$ is large enough. By plugging $\beta$ into \eqref{convexineq} and sending $K\to\infty$, we find that
\[
\frac{1}{p}\|\rho_{t_1}\|_p^p+ \frac{4(p-1)}{p^2}\int_{\delta_1}^{t_1}\int_{\bbR^3}|\nabla \rho_{s}^{p/2}|^2 dxds \le \frac{1}{p}\|\rho_{\delta_1}\|_p^p.
\]

Note that we know $\rho\in L^1(0, T; L^3(\bbR^3))$ and that $0<\delta\le \delta_1$ are arbitrary, we then find that
\[
\rho\in L_{\loc}^{\infty}((0, T); L^p(\bbR^3)),~p\in (1, 3].
\]

\begin{remark}
If we instead have $\rho\in L^2(0, T; L^2(\bbR^3))$, we will have $\rho\in L^{\infty}(0, T; L^p)$ for any $p\in [1,\infty)$.
\end{remark}

{\bf Step 2} Weak-strong uniqueness.

For $t>0$, we have $\frac{d}{dt}\|\rho\|_2^2+2\|\nabla\rho\|_2^2\le 0$. Since we do not have $\rho\in L^{\infty}(0, T; L^{3/2})$, so it is very hard to obtain the usual hyper-contractivity
\[
\sup_{0< t\le T}t^{p-\frac{3}{2}}\|\rho\|_p^p\le C.
\]
Using $\rho\in L^{\infty}(0, T; L^1(\bbR^d))$, we only have
\[
\frac{d}{dt}\|\rho\|_2^2+\|\rho\|_2^{10/3}\le 0
\]
which yields $\sup_{0<t\le T}t^{3/2}\|\rho\|_2^2\le C$. This is not enough to prove the uniqueness as in the standard Keller-Segel equations (see \cite{liuwang2016}). Hence, we must use other methods to prove the uniqueness. Here, we use the strategy in \cite{serfaty2018mean}, where weak-strong uniqueness was shown by using the Coulomb energy. (If we know $\rho\in L^2(0, T; L^2)$, then $\rho\in L^{\infty}(0, T; L^p)$ for any $p\in [1,\infty)$, and the usual method will work.)

Using $\rho\in L_{\loc}^{\infty}((0, T); L^p(\bbR^3)),~p\in (1, 3]$ and converting the semigroup  equation for $\rho$ into the semigroup form for $\nabla h$, we see $\nabla h\in AC(t_0, T; L^2)$ for any $t_0>0$. Here ``AC" means absolutely continuous. It is then clear that
\begin{gather}
\partial_t\nabla h=\nabla g*\nabla(\rho \nabla h)+\Delta \nabla h.
\end{gather}

Consider the strong solution $\rho_2$. $\nabla h_2\in C^{\infty}$ is clear.
\[
\nabla h_2=\nabla g*\rho_2=\nabla g|_{|x|\le 1}*\rho_2
+\nabla g|_{|x|\ge 1}*\rho_2.
\]
We have
\[
|\nabla h_2|\le \|\nabla g|_{|x|\le 1}\|_1\|\rho_2\|_{\infty}
+\|\nabla g_{|x|\ge 1}\|_{\infty}\|\rho_2\|_1.
\]
Hence, $\nabla h_2$ is bounded. The derivatives of $\nabla h_2$ are clearly bounded. Moreover, by the Hardy-Littlewood-Sobolev inequality, $\nabla h_2\in L^{\infty}(0,T; L^q(\bbR^3))$ for all $q>3/2$. Hence, $\rho_2\nabla h_2\in L^{\infty}(L^p)$
for any $p\ge 1$. We also have the equation
\[
\partial_t\nabla h_2=\nabla g*\nabla(\rho_2 \nabla h_2)+\Delta \nabla h_2.
\]
Note that the singular integral theory tells us that $\nabla g*\nabla(\rho_2 \nabla h_2)\in L^{\infty}(0, T; L^q)$, $q>1$.

For $t>0$, we have
\[
\partial_t(\nabla h-\nabla h_2)
=\nabla g*\nabla(\rho (\nabla h-\nabla h_2))
+\nabla g*\nabla((\rho-\rho_2)\nabla h_2)
+\Delta(\nabla h-\nabla h_2).
\]

For $t>t_0>0$, $\nabla h-\nabla h_2\in L^{\infty}[t_0, T, L^6]$, then we can pair the equation with
$\nabla h-\nabla h_2$.

Then,
\begin{multline*}
\frac{1}{2}\partial_t\|\nabla h-\nabla h_2\|_2^2
=-\langle \nabla h-\nabla h_2,
\rho(\nabla h-\nabla h_2)\rangle
-\langle \nabla h-\nabla h_2,
(\rho-\rho_2)\nabla h_2\rangle \\
+\langle \nabla h-\nabla h_2, \Delta(\nabla h
-\nabla h_2)\rangle=:I_1+I_2+I_3.
\end{multline*}
Here, we have use the fact
\begin{gather*}
\langle \nabla \phi , \nabla g*\nabla\cdot v\rangle
=-\langle  g*\Delta \phi, \nabla\cdot v\rangle=\langle \nabla (g*\Delta\phi), v\rangle=-\langle \nabla\phi, v\rangle.
\end{gather*}

Clearly, $I_1\le 0$ and $I_3\le 0$.
\begin{multline*}
I_2=-\int_{\bbR^3} \mathrm{div}(\nabla(h-h_2)\otimes\nabla(h-h_2)-\frac{1}{2}|\nabla(h-h_2)|^2 I)\nabla h_2\,dx \\
=\int_{\bbR^3} (\nabla(h-h_2)\otimes\nabla(h-h_2)-\frac{1}{2}|\nabla(h-h_2)|^2 I)\nabla^2h_2\,dx
\le C\int_{\bbR^3}|\nabla(h-h_2)|^2\,dx.
\end{multline*}
Note that it is exactly at this point we need $\rho_2$ to be the strong solution.

Using Gr\"onwall, we have for $t\in [t_0, T]$
\[
\|\nabla h(t)-\nabla h_2(t)\|_2^2\le C(T)\|\nabla h(t_0)-\nabla h_2(t_0)\|_2^2.
\]

Finally,
\[
\|\nabla h(t_0)-\nabla h_2(t_0)\|_2^2=\mathscr{E}(t_0)-2\int_{\bbR^3}\nabla h(t_0)\cdot\nabla h_2(t_0)\,dx
+\mathscr{E}_2(t_0),
\]
where
\[
\mathscr{E}(t):=\iint_{\bbR^3\times\bbR^3} g(x-y)\rho(x,t)\rho(y, t)\,dxdy=\int_{\bbR^3}|\nabla h(x, t)|^2\,dx,
\]
and $\mathscr{E}_2$ is similarly defined for $\rho_2$ and $h_2$. The second term is equal to $2\int_{\bbR^3} h_2(t_0)\rho(t_0)\,dx$ which is continuous in $t_0$. $\mathscr{E}_2(t)$ is also continuous. Hence, we then have
\[
\lim_{t_0\to 0}\|\nabla h(t_0)-\nabla h_2(t_0)\|_2^2
=\lim_{t_0\to 0}\mathscr{E}(t)-2\mathcal{E}(\rho_0)+\mathcal{E}(\rho_0)\le \mathcal{E}(\rho_0)-2\mathcal{E}(\rho_0)+\mathcal{E}(\rho_0)
\]
by the condition.  This means $\rho-\rho_2=0$ in $L^{\infty}(0, T; H^{-1})$.
Since they are both in $L^{\infty}(0, T; L^1)$, then they must equal almost everywhere.
\end{proof}

\bibliographystyle{unsrt}
\bibliography{meanfield}
\end{document}